\documentclass[11pt]{amsart}

\usepackage[pdftex]{graphicx}%
\usepackage[dvipsnames]{xcolor}

\usepackage{a4wide}

\usepackage{amsfonts}
\usepackage{amsmath}
\usepackage{amssymb}
\usepackage{amsthm}
\usepackage{mathtools}
\usepackage{tikz}
\usetikzlibrary{patterns}

\usepackage{paralist}
\usepackage[colorlinks,cite color=blue,pagebackref=true,pdftex]{hyperref}

\usepackage[T1]{fontenc}

\usepackage{caption}
\usepackage{subcaption}

\newcommand\Defn[1]{\textbf{\color{black}#1}}
\newcommand\Def[1]{\Defn{#1}}

\renewcommand\u{\mathbf{u}}
\renewcommand\v{\mathbf{v}}
\newcommand\w{\mathbf{w}}

\newcommand\q{\mathbf{q}}

\newcommand\polar{\vee}

\newcommand\eps{\varepsilon}
\renewcommand\emptyset{\varnothing}
\newcommand\Z{\mathbb{Z}}

\newcommand\R{\mathbb{R}}

\newcommand\C{\mathbb{C}}
\newcommand\z{\mathbf{z}}
\newcommand\x{\mathbf{x}}

\newcommand\inner[1]{\langle {#1} \rangle}
\newcommand\defeq{\coloneqq}

\newcommand\CAb{\beta}%
\newcommand\bExt{\widecheck{\CAb}}%
\newcommand\Sphere{\mathbb{S}}

\newcommand\op{\mathrm{op}}

\DeclareMathOperator{\sgn}{sgn}
\DeclareMathOperator{\relint}{relint}
\DeclareMathOperator{\interior}{int}
\DeclareMathOperator{\aff}{aff}
\newcommand\affL{\aff_0}

\DeclareMathOperator{\cpl}{cpl}

\DeclareMathOperator{\lineal}{lineal}

\DeclareMathOperator{\cone}{cone}
\DeclareMathOperator{\vol}{vol}
\DeclareMathOperator{\rk}{rk}

\DeclareMathOperator{\abR}{R}
\DeclareMathOperator{\abE}{E}
\DeclareMathOperator{\abP}{Pr}
\DeclareMathOperator{\abM}{M}

\DeclareFontFamily{U}{mathx}{\hyphenchar\font45}
\DeclareFontShape{U}{mathx}{m}{n}{ <5> <6> <7> <8> <9> <10> <10.95> <12> <14.4>
  <17.28> <20.74> <24.88> mathx10 }{}
\DeclareSymbolFont{mathx}{U}{mathx}{m}{n}
\DeclareFontSubstitution{U}{mathx}{m}{n}
\DeclareMathAccent{\widecheck}{0}{mathx}{"71}

\newtheorem{thm}{Theorem}[section] \newtheorem{cor}[thm]{Corollary}
\newtheorem{lem}[thm]{Lemma} \newtheorem{prop}[thm]{Proposition}

\theoremstyle{definition} \newtheorem{definition}[thm]{Definition}
 \newtheorem{rem}[thm]{Remark}

\title{Generalized angle vectors, geometric lattices, and flag-angles}

\author{Spencer Backman}

\address{Department of Mathematics and Statistics,
University of Vermont, Burlington,
Vermont, USA.}
\email{Spencer.Backman@uvm.edu}

\author{Sebastian Manecke} 
\author{Raman Sanyal}

\address{Institut f\"ur Mathematik, Goethe-Universit\"at Frankfurt, Germany} 
\email{manecke@math.uni-frankfurt.de}
\email{sanyal@math.uni-frankfurt.de}



\date{\today}

\begin{document}

%
%
%

\begin{abstract}
    Interior and exterior angle vectors of polytopes capture curvature
    information at faces of all dimensions and can be seen as metric
    variants of $f$-vectors. In this context, Gram's relation takes the place of the
    Euler--Poincar\'e relation as the unique linear relation among interior
    angles. We show the existence and uniqueness of Euler--Poincar\'e-type
    relations for generalized angle vectors
 by building a bridge to the
    algebraic combinatorics of geometric lattices, generalizing work of
    Klivans--Swartz. 

    We introduce flag-angles of polytopes as a geometric counterpart to
    flag-$f$-vectors. Flag-angles generalize the angle deficiencies of
    Descartes--Shephard, Grassmann angles, and spherical intrinsic volumes.
    Using the machinery of incidence algebras, we relate flag-angles of
    zonotopes to flag-$f$-vectors of graded posets. This allows us to determine
    the linear relations satisfied by interior/exterior flag-angle
    vectors.

\end{abstract}
\maketitle


\newcommand\Tcone[2]{\mathrm{T}_{#1}{#2}}%
\newcommand\SCA{\nu}%
\newcommand\intSCA{\widehat{\SCA}}%
\newcommand\Ncone[2]{\mathrm{N}_{#1}#2}%
\newcommand\extSCA{\widecheck{\SCA}}%
\newcommand\CA{\alpha}%
\newcommand\aInt{\widehat{\CA}}%
\newcommand\aExt{\widecheck{\CA}}%
\newcommand\Cones{\mathcal{C}}%
\newcommand\LL{\mathsf{L}}%
\newcommand\Lflats{\mathcal{L}}%
\newcommand\Ltop{\mathbf{1}}%
\newcommand\Lbot{\mathbf{0}}%
\newcommand{\Incidence}{\mathcal{I}}%
\newcommand\Arr{\mathcal{H}}%
\newcommand\Lfaces{\mathcal{F}}%
\newcommand\poset{\mathcal{P}}%
\newcommand\FP{\mathcal{F}}%
\newcommand\FW{\mathbf{W}}%
\newcommand\faInt{\widehat{\boldsymbol{\CA}}}%
\newcommand\faExt{\widecheck{\boldsymbol{\CA}}}%

\section{Introduction}\label{sec:intro}
For a convex polytope $P \subset \R^d$ of dimension $d$, let $f_i(P)$ be the
number of $i$-dimensional faces of $P$ for $i = 0,1,\dots,d-1$. The Euler--Poincar\'e
relation states that  the face numbers satisfy
\begin{equation}\label{eqn:euler}
    f_0(P) - f_1(P) + f_2(P) - \cdots + (-1)^{d-1} f_{d-1}(P) \ = \ 1 -
    (-1)^d \, .
\end{equation}
This simple linear relation among the face numbers is the key to a rich
interplay of geometry, combinatorics, and algebra as amply demonstrated, for
example, in~\cite{RotaKlain, Stanley-eulerian}. In particular, the
Euler--Poincar\'e relation~\eqref{eqn:euler} is, up to scaling, the only
linear relation among the face numbers of polytopes of fixed dimension $d$;
cf.~\cite{hohn}, \cite[Sect.~8.1]{gruenbaum}.  In this paper, we will show 
the existence and uniqueness of Euler--Poincar\'e-type relations for certain
class of \emph{semi-discrete} invariants and their interplay with algebraic
combinatorics. These invariants are generalizations of the well-known
\emph{interior} and \emph{exterior} angle vectors of polytopes.

The local geometry of $P$ at a face $F$ is captured by the \Def{tangent cone}
(or \Def{inner cone}) $\Tcone{F}{P} \defeq \cone(-F + P)$, which is the cone of
feasible directions from $F$. The \Def{interior angle} of $P$ at $F$ is the
spherical volume
\[
    \intSCA(F,P) \ \defeq \ \SCA(\Tcone{F}{P}) \ = \
    \frac{\vol(\Tcone{F}{P} \cap B_d)}{\vol(B_d)} \, ,
\]
which generalizes the notion of dihedral angle to faces of all dimensions.
The $i$-th interior angle of $P$ is given by $\intSCA_i(P) \defeq \sum_F
\intSCA(F,P)$, where $F$ ranges over all faces of dimension $i$. The
\Def{interior angle vector} $\intSCA(P) = (\intSCA_i(P))_{i=0,\dots,d-1}$
captures curvature information at faces of various dimensions and combines
combinatorial as well as metric properties of $P$.  The fundamental relation,
called \Def{Gram's relation}, which is satisfied by interior angle vectors of
$d$-dimensional polytopes is 
\begin{equation}\label{eqn:SCAgram}
        \intSCA_0(P) - \intSCA_1(P) + \intSCA_2(P) - \cdots + (-1)^{d-1}
        \intSCA_{d-1}(P) \ = \ (-1)^{d+1} \, .
\end{equation}
We refer to Gr\"unbaum~\cite[Sect.~14.4]{gruenbaum} for a detailed 
historic account of Gram's relation and its importance.

Gromov and Milman~\cite{GM87} introduced and studied an \emph{anisotropic}
notion of angle of a cone $C \subset \R^d$
\[
    \SCA_K(C) \ \defeq \ \frac{\vol(C \cap K)}{\vol(K)} \, ,
\]
where $K \subset \R^d$ is a fixed centrally-symmetric convex body. These
\emph{cone (probability) measures} are related to surface measures of general
(star) convex bodies~\cite{NR03, Naor07,BGMN05}. Using the integral-geometric
perspective developed by Perles--Shephard~\cite{PS67} (see also~\cite{Welzl}),
one easily shows that Gram's relation is the essentially unique linear
relation for the associated interior angles $\intSCA_{K}(P)_i = \sum_F
\SCA_K(\Tcone{F}{P})$ of $d$-dimensional polytopes.
In the first part of the paper, we prove the existence and uniqueness of
Gram's relation for a more general class of cone angles by way of an
astounding connection to the combinatorics of posets and matroids.

A map $\CA : \Cones_d \to R$ from the collection of convex polyhedral cones
$\Cones_d$ in $\R^d$ into some unital ring $R$ is a \Def{valuation} if
$\CA(\{0\}) = 0$ and
\[
    \CA(C \cup C') \ = \ \CA(C) + \CA(C') - \CA(C \cap C') \,
\]
for all $C,C' \in \Cones_d$ such that $C \cup C', C \cap C' \in \Cones_d$.  A
valuation $\CA$ is \Def{simple} if $\CA(C) = 0$ whenever $\dim C < d$ and we
call $\CA$ a \Def{cone angle} if in addition $\CA(\R^d) = 1$. Cone valuations
play a decisive role in integral geometry~\cite{SchneiderWeil} and cone angles
strictly subsume cone probability measures; see Section~\ref{sec:CA}. Note
that we do not require cone angles to be rotationally invariant or to satisfy
any positivity conditions.  The \Def{interior $\CA$-angle} of a polytope $P$
at a face $F \subseteq P$ is then defined as $\aInt(F,P) \defeq
\CA(\Tcone{F}{P})$. We also define the \Def{exterior $\CA$-angle} of $P$ at
$F$ as $\aExt(F,P) \defeq \CA(\Ncone{F}{P} + \affL(F))$, where $\affL(F)$ is the
linear subspace parallel to $F$ and $\Ncone{F}{P}$ is the \Def{normal cone} of
$P$ at $F$, that is, the cone polar to $\Tcone{F}{P}$.  As expected, the
interior and exterior \Def{$\CA$-angle vector} of $P$ are defined through
\[
    \aInt_i(P) \ \defeq \ \sum_{F} \aInt(F,P) \qquad \text{ and } \qquad
    \aExt_i(P) \ \defeq \ \sum_{F} \aExt(F,P) \, ,
\]
where the sums are over all faces $F \subseteq P$ of dimension $i$, for $i =
0, \dots, d-1$. Our first main result is this.

\begin{thm}\label{thm:gram}
    Let $\CA$ be a cone angle. Then, up to scaling, the only linear
    relations satisfied by $\aInt(P)$, respectively $\aExt(P)$, for
    any $d$-dimensional polytope $P$ are
    \begin{equation} \label{eqn:gram} 
        \aInt_0(P) - \aInt_1(P) + \aInt_2(P) - \cdots + (-1)^{d-1}
        \aInt_{d-1}(P) \ = \ (-1)^{d+1} \, ,
    \end{equation}%
    \begin{equation}
      \label{eqn:ext_gram}
        \aExt_0(P) \ = \ \sum_{v} \aExt(v,P) \ = \ 1 \,.
    \end{equation}    
\end{thm}
Showing the validity of both relations is not difficult. Indeed,
in~\cite{PS67} a proof of~\eqref{eqn:SCAgram} is sketched that works for
general cone angles. For completeness, we give a proof using a conical version
of the Brianchon--Gram relation of~\cite{AS15}; see
also~\cite{lawrence,shephard-elem}. The main challenge in proving
Theorem~\ref{thm:gram} is uniqueness, as none of the analytic and geometric
properties of $\SCA_K$ carry over to general cone angles.  We prove
Theorem~\ref{thm:gram} by establishing a powerful connection between the
geometry and the combinatorics of zonotopes.

To explain the combinatorial connection, define $\LL(F) \defeq \affL(F)^\perp$ for any
non-empty face $F \subseteq P$ and let $\Lflats(P) \defeq \{\LL(F) : \emptyset \neq F
\subseteq P \}$ partially ordered by \emph{reverse} inclusion. This is a finite graded
poset of rank $d$. In particular, if $P$ is a zonotope, that is, a Minkowski-sum of
segments, then $\Lflats(P)$ is a geometric lattice, called \Defn{lattice of flats}. The
\Def{Whitney numbers} of the first kind $w_i$ and of the second kind $W_i$ of a graded
poset are important enumerative invariants~\cite[Sect.~3.10]{EC1}, whose precise
definition we recall in Section~\ref{sec:zono}. The following result allows us to show
the uniqueness in Theorem~\ref{thm:gram} on a purely combinatorial level.

\begin{thm}\label{thm:whitney}
    Let $Z \subset \R^d$ be a $d$-dimensional zonotope with lattice of flats
    $\Lflats = \Lflats(Z)$. For any cone angle $\CA$ we have
    \[
        \aExt_i(Z) \ = \ W_{i}(\Lflats(Z)) \quad \text{ and } \quad
        \aInt_i(Z) \ = \ (-1)^{d-i} w_{d-i}(\Lflats(Z)^\op)
    \]
    for all $i=0,\dots,d-1$.
\end{thm}

For the standard cone angle, the second equation in Theorem~\ref{thm:whitney}
was proven in Klivans--Swartz~\cite{KlivansSwartz} using ideas similar to
those in~\cite{PS67} involving projections of zonotopes.
Theorem~\ref{thm:whitney} is then used to show that interior/exterior angle
vectors of zonotopes certify the uniqueness of \eqref{eqn:gram} and
\eqref{eqn:ext_gram}. 
Theorem~\ref{thm:whitney} and Theorem~\ref{thm:gram} are proved in
Section~\ref{sec:zono}. 

In Section~\ref{sec:incalg}, we recast the correspondence between
interior/exterior angles and Whitney numbers of the first and second kind
in algebraic terms.  McMullen~\cite{McM-angle} showed that
$\aInt,\aExt$ can be interpreted as elements in the incidence algebra
$\Incidence(\Lfaces(P))$ of the face poset $\Lfaces(P)$ of $P$. In
Section~\ref{sec:incalg}, we show that for a zonotope $Z$ with lattice
of flats $\Lflats$, there is a subalgebra
$\Incidence_\LL \subseteq \Incidence(\Lfaces(Z))$ such that the map
$F \mapsto \LL(F)$ yields a ring map
$ \LL_* : \Incidence_\LL \to \Incidence(\Lflats)$. We show that $\aExt \in \Incidence_\LL$ and
$\LL_* \aExt = \zeta_{\Incidence(\Lflats)}$.  McMullen's
\emph{inverse} angles~\cite{McM-polytopealgebra} then allow us to show
$\LL_* \aInt' = \mu_{\Incidence(\Lflats)}$, where $\aInt'$ is a slight
modification of $\aInt$. This yields an elegant algebraic proof of
Theorem~\ref{thm:whitney} and explains the appearance of Whitney
numbers of the first and second kind. In addition, we give a simple
proof of a beautiful relation due to Klivans and
Swartz~\cite{KlivansSwartz} between spherical intrinsic volumes of a
zonotope $Z$ and the characteristic polynomial of $\Lflats(Z)$; see
Corollary~\ref{cor:alg-KS}.

The second goal of the paper is to introduce \emph{flag-angle vectors} as a
unifying geometric concept and to exhibit and exploit parallels to the theory
of flag-vectors of posets. To motivate flag-angle vectors, let $P$ be a 
$3$-dimensional polytope. Descartes defined the \emph{angle defect}
at a vertex $v$ as $\delta(v,P) = 1 - \sum_F \intSCA(v,F)$, where the sum is
over all $2$-dimensional faces $F$ containing $v$ and he showed that
\[
    \sum_v \delta(v,P) \ = \ 2 \, .
\]
For a $d$-polytope $P$ and $i=0,\dots,d-3$, Shephard~\cite{Shephard-angle}
defines the \emph{$i$-th total angle deficiency}
\[
    \delta_i(P) \defeq \ f_i(P) - \sum_{G \subset F} \intSCA(G,F) \, ,
\]
where the sum is over all faces $G \subset F$, with $\dim G = i$ and $\dim F =
d-1$. Generalizing Descartes' result, Shephard showed that
\begin{equation}\label{eqn:angle_defect}
    \delta_0(P) - \delta_1(P) + \dots + (-1)^{d-3} \delta_{d-3}(P) \ = \ 1 +
    (-1)^{d-1}\,.
\end{equation}
Such generalized Descartes-relations were further studied for
manifolds~\cite{GrunbaumShephard} and in relation with stratified
curvature~\cite{Bloch} and polyhedral Gauss--Bonnet
theorems~\cite{Schneider-GaussBonnet}.

In contrast to interior angle vectors, total angle deficiencies
record the interaction of faces of various dimensions and the generalized
Descartes--relation~\eqref{eqn:angle_defect} shows that angle deficiencies are
not independent. In a different direction, McMullen~\cite{McM-angle} showed
that exterior angles can be computed from interior angles of flags of faces
\begin{equation}\label{eqn:int_to_ext}
    (-1)^d \widecheck{\SCA}_i(P) = \sum_{F_1 \subset F_2 \subset \cdots
    \subset F_k} (-1)^{k+1} 
    \intSCA(F_1,F_2) 
    \intSCA(F_2,F_3) \cdots
    \intSCA(F_{k},P)  \, ,
\end{equation}
where the sum is over all flags of faces with $\dim F_1 = i$. Moreover,
McMullen showed that various other measures of curvature, such as spherical
intrinsic volumes and Grünbaum's \emph{Grassmann
angles}~\cite{Grunbaum-grassmann} can be computed from chains of
interior (or exterior) angles.

\begin{definition}
    Let $\CA$ be a cone angle. For a $d$-dimensional polytope $P$ and a
    non-empty set $S = \{ 0 \le s_1 < s_2 < \cdots < s_k \le d-1 \}$, define the 
    \Def{interior flag-angle} by
    \begin{equation}\label{eqn:flag}
        \aInt_S(P) \ \defeq \ \sum_{F_1 \subset F_2 \subset \cdots \subset F_k}
        \aInt(F_1,F_2) \,  \aInt(F_2,F_3) \cdots \aInt(F_k,P) \, ,
    \end{equation}
    where the sum is over all chains of faces of $P$ such that $\dim F_i =
    s_i$ for $i=1,\dots,k$. The \Def{exterior flag-angle} $\aExt_S(P)$ is
    defined analogously and we set $\aExt_\emptyset(P) \defeq
    \aInt_\emptyset(P) \defeq 1$. 
\end{definition}

The vectors $\faInt(P) = (\aInt_S(P))_S$ and $\faExt(P) = (\aExt_S(P))_S$ are
called the interior and exterior \Def{flag-angle vectors} of $P$.
%
%
In Sections~\ref{sec:flag} and~\ref{sec:flag-whitney}, we determine the affine
spaces spanned by interior and exterior flag-angle vectors, respectively. 

\begin{thm}\label{thm:flag_rels}
    Let $P$ be a $d$-dimensional polytope and $S \subseteq [d-1] \defeq \{1, 2, \dots, d-1\}$.  
    For any cone angle $\CA$, we have 
    \[
        \aExt_{S}(P) \ = \ \aExt_{S \cup \{0\}}(P) 
        \qquad \text{ and } \qquad
        \sum_{i=0}^{t-1} (-1)^{i} \aInt_{S \cup \{i\}}(P) \ = \ (-1)^{t+1}
        \aInt_S(P) \, ,
    \]
    where $t = \min(S \cup \{d\})$.

    Moreover, the affine hull of exterior flag-angles as well as the
    affine hull of interior flag-angles is of dimension $2^{d-1} -
    1$. These spaces are spanned by the flag-angle vectors of
    zonotopes.
\end{thm}

Since $\delta_i(P) = f_i(P) - 2 \intSCA_{0,d-1}(P)$,
Theorem~\ref{thm:flag_rels} for $S = \{d-1\}$ together with the
Euler--Poincar\'e relation~\eqref{eqn:euler} implies Shephard's
result~\eqref{eqn:angle_defect}. Theorem~\ref{thm:flag_rels} also determines
the spaces of linear relations on all specializations of flag-angles, including
spherical intrinsic volumes and Grassmann angles.

Flag angle vectors are \emph{semi-discrete} counterparts to \emph{flag
vectors} of polytopes and posets. Let $\poset$ be a graded poset of rank $d+1$
with minimal element $\Lbot$ and maximal element $\Ltop$. For $S \subseteq [d]$, the number of chains of elements
\[
     \Lbot \ \prec_\poset \ c_1 \ \prec_\poset \ c_2 \ \prec_\poset \ \cdots \
     \prec_\poset \ c_k \ \prec_\poset \ \Ltop
\]  
such that $S = \{ \rk c_1, \rk c_2, \dots, \rk c_k \}$ is called the
\Def{flag-Whitney number} of the \Def{second kind} $W_S(\poset)$.  The resulting vector
$\FW(\poset)=(W_S(\poset))_{S}$ is commonly known as the
\Def{flag-vector} of $\poset$. Flag-vectors of face posets of polytopes or,
more generally, of Eulerian posets have received considerable attention,
starting with the seminal paper Bayer--Billera~\cite{BB}.  Bayer and Billera
determined the linear relations on flag-vectors of Eulerian posets, called the
\emph{generalized Dehn-Sommerville relations} and showed that flag-vectors
of Eulerian posets of rank $d+1$ span an affine space of dimension $F_d$,
where $F_d$ is the $d$-th Fibonacci number.
Billera--Ehrenborg--Readdy~\cite{BER} showed that these spaces are spanned by
the flag vectors of $d$-dimensional zonotopes.

To complete the relation to flag vectors, we introduce the \Defn{flag-Whitney
numbers} of the \Def{first kind} $w_S(\poset)$, which extend the ordinary
Whitney numbers $w_i(\poset)$ to flags. They are obtained via
inclusion-exclusion from the flag-vector and are complementary to the usual
\emph{flag $h$-vector}.  The following result extends
Theorem~\ref{thm:whitney} to flag-angle vectors. We set $d - S \defeq \{ d - s :
s \in S \}$.

\begin{thm}\label{thm:flag_whitney}
    Let $\CA$ be a cone angle and $Z \subset \R^d$ a full-dimensional
    zonotope. Then for a nonempty $S \subseteq \{0,1,\dots,d-1\}$, we have
    \[
        \aExt_S(Z) \ = \ W_{S}(\Lflats(Z)) \quad \text{ and } \quad
        \aInt_S(Z) \ = \ (-1)^{d-r} w_{d - S}(\Lflats(Z)^\op) \, ,
    \]
    where $r = \min(S)$. 
\end{thm}


We prove Theorems~\ref{thm:flag_rels} and~\ref{thm:flag_whitney} in
Section~\ref{sec:flag} using the algebraic tools developed in
Section~\ref{sec:incalg}. As before, Theorem~\ref{thm:flag_whitney} is the key
to proving the main statement of Theorem~\ref{thm:flag_rels}.
Corollary~\ref{cor:int_to_ext} extends McMullen's
relation~\eqref{eqn:int_to_ext} to flag-angles, which then shows that it
suffices to only consider exterior flag-angle vectors.  This amounts to
showing that there are no linear relations on flag-vectors of $\Lflats(Z)$ for
$d$-dimensional zonotopes $Z$, which we do in Theorem~\ref{thm:zono_span}.
This strengthens a result of Billera--Hetyei~\cite{BilleraHetyei}. Since the
flag-vector of a zonotope is encoded by the flag-vector of its lattice of
flats by results in~\cite{BayerSturmfels}, Theorem~\ref{thm:zono_span} also
strengthens the main result in Billera--Ehrenborg--Readdy~\cite{BER} in that
flag-vectors of Eulerian posets are spanned by the flag vectors of zonotopes.
In Section~\ref{sec:spherical_intrinsic_volumes} we revisit Grassmann-angle
and spherical intrinsic volumes from the perspective of Crofton-type formulas
and introduce a generalization for cone angles. This allows us to generalize
the main result of Gr\"unbaum's fundamental paper~\cite{Grunbaum-grassmann} on Grassmann angles of
polytopes.

\textbf{Acknowledgements.}
Research that led to this paper was supported by the
DFG-Collabora\-tive Research Center, TRR 109 ``Discretization in Geometry and
Dynamics'' and by the National Science Foundation under Grant No.~DMS-1440140
while the authors were at the Mathematical Sciences Research Institute in
Berkeley, California, during the Fall 2017 semester on \emph{Geometric and
Topological Combinatorics}.  S.~Backman was also supported by a Zuckerman STEM
Postdoctoral Scholarship. We thank Marge Bayer, Curtis Greene, Carly Klivans,
and Richard Ehrenborg for helpful discussions.

\section{Cone angles and linear relations}\label{sec:CA}

In this section, we recall the basic geometric constructions and prove the
existence of the relations stated in Theorem~\ref{thm:gram}.

Let $P \subset \R^d$ be a full-dimensional polytope and $q \in P$. The
\Def{tangent cone} of $P$ at $q$ is
\[
    \Tcone{q}{P} \ \defeq \{ u \in \R^d : q + \eps u \in P \text{ for some }
    \eps > 0 \} \ = \ \cone(-q + P) \, .
\]
It is easy to verify that $\Tcone{q}{P} = \Tcone{q'}{P}$ if and only if $q,q'
\in \relint(F)$ for some face $F \subseteq P$ and we recover $\Tcone{F}{P} =
\Tcone{q}{P} = \cone(-F + P)$. Tangent Cones capture the local geometry of $P$
at $F$ in that faces of $\Tcone{F}{P}$ are in correspondence to faces $G
\supseteq F$ under the correspondence $G \mapsto \Tcone{F}{G}$. If $P$ is not
full-dimensional, we extend the definition to
\[
    \Tcone{F}{P} \ = \ \cone(-F + P) + \affL(P)^\perp \, .
\]

\newcommand\Ocone[2]{\mathrm{O}_{#1}#2}%
The \Def{normal cone} of a face $F \subseteq P$ is the polyhedral cone
\[
    \Ncone{F}{P} \ \defeq \ \{ c \in \R^d \ : \ \inner{c,x} \ge \inner{c,y}
    \text{ for all } x \in F, y \in P \} \, .
\]
By construction, $\Ncone{F}{P}$ is contained in $\affL(F)^\perp$. We define the
\Def{outer cone} of $P$ at $F$ as the $d$-dimensional cone
\[
    \Ocone{F}{P} \ \defeq \ \Ncone{F}{P} + \affL(F) \, .
\]

As stated in the introduction, a cone angle $\CA : \Cones_d \to R$ is a simple
valuation normalized so that $\CA(\R^d)=1$. Cone probability measures
$\SCA_K(C) = \frac{\vol(C \cap K)}{\vol(K)}$ for $K$ a full-dimensional convex
body are cone angles, but the notion of cone angles is richer. For example,
for any point $q \in \R^d$, let $B_\eps(q)$ be the ball with radius $\epsilon
> 0$ centered at $q$. Then
\[
    \omega_q(C) \ \defeq \ \lim_{\eps \to 0} \frac{\vol(B_\eps(q) \cap
    C)}{\vol(B_\eps(q))}
\]
defines a cone angle which, for $q \neq 0$ does not come from a measure.
Further instances can be obtained, for example, from the construction of
Dehn--Hadwiger functionals~\cite[Sect.~2.2.2]{hadwiger}.  

The following standard construction yields a \emph{universal} cone valuation.
Let $\Z\Cones_d$ be the free abelian group with generators $e_C$ for $C \in
\Cones_d$. Let $\mathcal{U} \subset \Z\Cones_d$ be the subgroup generated by
\[
    e_{C \cup D} + e_{C \cap D} - e_{C} - e_{D}
\]
\newcommand\CG{\mathbb{S}}%
for all Cones $C,D \in \Cones_d$ such that $C \cup D, C \cap D \in \Cones_d$
and let $\mathcal{S}$ be the subgroup generated by all elements $e_C$ for
which $\dim C < d$. We call $\CG \defeq \Z\Cones_d / (\mathcal{U} +
\mathcal{S})$ the \Def{simple cone group}. Clearly, if $\phi' : \CG \to
\R$ is additive, then $\phi(C) \defeq \phi'(e_C)$ defines a valuation.
Volland~\cite{volland} essentially showed that every valuation lifts to
a homomorphism on $\CG$. We record this as follows.
\begin{thm}[Volland]
    The map $\Cones_d \to \CG$ given by $C \mapsto e_C$ is the universal
    cone valuation.
\end{thm}

By work of Gr\"omer~\cite{gromer} we can identify elements in $\Z\Cones_d /
\mathcal{U}$ with linear combinations of indicator functions $f =
\sum_{i=1}^ka_i [C_i]$ where $C_1,\dots,C_k \in \Cones_d$ and $a_1,\dots,a_k
\in \Z$. 

\begin{cor}\label{cor:almost}
    Let $f = \sum_i a_i[C_i]$ and $f' = \sum_i a'_i[C'_i]$. Then $f = f'$ in
    $\CG$ if and only if $f(p) = f'(p)$ for almost all $p \in \R^d$.
\end{cor}

\newcommand\Int[1]{\mathbb{T}_{#1}}
\newcommand\Ext[1]{\mathbb{O}_{#1}}
We will make extensive use of this correspondence.  In particular, we
can define the universal interior and exterior angle vectors. We will
describe them a little differently. Let $\CG[t]$ be the abelian group of
formal polynomials in $t$ with coefficients in $\CG$. For a polytope
$P$, we define
\[
    \Int{P}(t) \ \defeq \ \sum_{F} [\Tcone{F}{P}] \, t^{\dim F} \quad \text{ and
    } \Ext{P}(t) \ \defeq \ \sum_{F} [\Ocone{F}{P}] \, t^{\dim F}  \, ,
\]
where in both cases the sum is over all nonempty faces $F \subset P$. Thus, if
$\CA$ is a cone angle, then $\aInt(P)$ is naturally identified with the
coefficients of $\CA(\Int{P})$. Here is the first benefit.

\begin{prop}\label{prop:Ext0}
    Let $\CA$ be a cone angle and $P \subset \R^d$ a
    full-dimensional polytope. Then $\aExt_0(P)  =  1$.
\end{prop}
\begin{proof}
    Note that $\Ocone{v}{P} = \Ncone{v}{P}$ for any vertex $v \in P$.  Now,
    for a general $c \in \R^d$, the linear function $x \mapsto \inner{c,x}$
    will be maximized at a unique vertex of $P$. Corollary~\ref{cor:almost}
    implies
    \[
        \sum_{v} [\Ocone{v}{P}] \ = \ [\R^d]
    \]
    as elements in $\CG$. Applying $\alpha$ to both sides of the equation
    finishes the proof.
\end{proof}

To complete the first half of Theorem~\ref{thm:gram}, recall that the
\Def{homogenization} of a polytope $P \subset \R^d$ is the polyhedral cone
\[
    \hom(P) \ \defeq \ \cone( P \times \{1\}) \ = \ \{ (x,t) \in \R^d \times \R
    : \ t \ge 0, x \in t P \}
\]
Every face $\{0\} \neq F' \subseteq \hom(P)$ is of the form $\hom(F)$ for some
nonempty face $F \subseteq P$. In particular, the definition of tangent Cones
extends to faces of $\hom(P)$. Moreover,
\[
    \Tcone{F}{P} \ \cong \ \Tcone{F'}{\hom(P)} \cap \{ (x,t) : t = 0 \} \, .
\]
Note that for $F' = \{0\}$, we have $\Tcone{F'}{\hom(P)} = \hom(P)$ and hence
\[
    \Tcone{F'}{\hom(P)} \cap \{ (x,t) : t = 0 \} \ = \ \{ (0,0) \}\,.
\]
On the other
hand, if $F' = C$, then $\Tcone{C}{C} = \R^{d+1}$.  
A Brianchon-Gram relation for polyhedral Cones was proved in~\cite{AS15}.

\begin{lem}[{\cite[Lem.~4.1]{AS15}}]\label{lem:AS}
    Let $C \subseteq \R^{d+1}$ be a full-dimensional cone. Then
    as functions on $\R^{d+1}$ 
    \[
        \sum_{F'} (-1)^{\dim F'} [\Tcone{F'}{C}] \ = \
        (-1)^{d+1} [\mathrm{int}(-C)] \, ,
    \]
    where the sum is over all nonempty faces $F' \subseteq C$ and
    $\mathrm{int}(-C)$ denotes the interior of $-C$.
\end{lem}

The following proposition proves the first half of Theorem~\ref{thm:gram}.

\begin{prop}\label{prop:Int0}
    Let $\CA$ be a cone angle on $\R^d$ and let $P \subset \R^d$ be a
    full-dimensional polytope. Then
    \[
        \aInt_0(P) - \aInt_1(P) + \aInt_2(P) - \cdots + (-1)^{d-1}
        \aInt_{d-1}(P) \ = \ (-1)^{d+1} \, .
    \]
\end{prop}

\begin{proof}
     Let $C = \hom(P) \subset \R^{d+1}$. This is a
     full-dimensional cone and Lemma~\ref{lem:AS} together with the restriction to $\R^d
     \times \{0\}$ and the preceding remarks yield the following relation on functions on
     $\R^d$
    \begin{equation}\label{eqn:gen_BR}
        [\{0\}] + \sum_{ F} (-1)^{\dim F+1}
        [\Tcone{F}{P}] + (-1)^{d+1} [\R^d] \ = \ (-1)^{d+1}[\emptyset] \, ,
    \end{equation}
    where the sum is over all nonempty faces $F \subseteq P$ with $F \neq P$.
    The above equation in $\CG$
    reads
    \[
        (-1)^{d+1} [\R^d] \ = \ 
        \sum_{\emptyset \neq F \subset P} (-1)^{\dim F}[\Tcone{F}{P}] 
        \ = \ \Int{P}(-1)
    \]
    and applying $\CA$ to both sides, yields the result.
\end{proof}

\section{Belt polytopes and angle vectors}\label{sec:zono}

A convex polytope $Z \subset \R^d$ is a \Def{zonotope} if there are
$z_1,\dots,z_k \in \R^d \setminus \{0\}$ and $t \in \R^d$ such that
\[
    t + Z \ = \ \sum_{i=1}^k [-z_i,z_i]
    \ = \ \{ \lambda_1 z_1 + \cdots + \lambda_k z_k : -1 \le
    \lambda_1,\dots,\lambda_k \le 1  \} \, .
\]
Zonotopes play an important role in geometric
combinatorics~\cite[Ch.7]{ziegler}) as well as convex
geometry~\cite{GoodeyWeil}.  Faces of zonotopes are zonotopes and hence all
$2$-dimensional faces of a zonotope are centrally-symmetric polygons. In fact,
this property characterizes zonotopes; see Bolker~\cite{Bolker}. A polytope $P
\subset \R^d$ is a \Def{belt polytope} (or \Def{generalized zonotope}) if and
only if every $2$-face $F \subset P$ has an even number of edges and opposite
edges are parallel. Belt polytopes were studied by Baladze~\cite{Baladze} (see
also~\cite{Bolker}) and are equivalently characterized by the fact that their
normal fans are induced by hyperplane arrangements. 

Let $\Arr$ be a central arrangement of hyperplanes, that is, the collection of
some oriented linear hyperplanes $H^0_i \defeq z_i^\perp$ for $i=1,\dots,k$.
We write $H_i^+ = \{ x : \inner{z_i, x} > 0\}$ and $H_i^-$ accordingly. For a
point $p \in \R^d$, let $\sigma_i = \sgn \inner{z_i,p}$ for $i=1,\dots,k$.
Then
\[
    H_\sigma \ \defeq \ H_1^{\sigma_1} \cap H_2^{\sigma_2} \cap \cdots \cap
    H_k^{\sigma_k}
\]
where $\sigma = (\sigma_1,\dots,\sigma_k) \in \{-,0,+\}^k$ is a relatively
open cone containing $p$. This shows that the collection $\{ H_\sigma : \sigma
\in \{-,0,+\}^k \}$ of relatively open Cones partitions $\R^d$. The
\Def{lattice of flats} of a hyperplane arrangement $\Arr$ is the collection
of linear subspaces 
\[
    \Lflats(\Arr) \ \defeq \ \{ H_{i_1}^0 \cap H_{i_2}^0 \cap \cdots  \cap
    H_{i_r}^0  : 1 \le i_1 < \cdots < i_r \le k, r \ge 0 \}
\]
partially ordered by \emph{reverse} inclusion. This is a graded lattice with
minimal element $\R^d$ and maximal element $H_1^0 \cap \cdots \cap H_k^0$.
For every relatively open cone $H_\sigma$, we have that
\[
    \affL(H_\sigma) \ = \ \bigcap_{i : \sigma_i = 0} H_i^0  
\]
is an element of $\Lflats(\Arr)$ and we record the following consequence.

\begin{prop}\label{prop:arr_part}
    Let $\Arr$ be an arrangement of hyperplanes with $\Lflats =
    \Lflats(\Arr)$. Then any $L \in \Lflats$ is partitioned by the collection
    of relatively open Cones $R$ of $\Arr$ with $\affL(R) \subseteq L$.
\end{prop}

Let $P \subset \R^d$ be a polytope of positive dimension and recall that for a
non-empty face $F \subseteq P$, we defined $\LL(F) = \affL(F)^\perp$.  We can
associate an arrangement $\Arr(P)$ with hyperplanes $\LL(e)$ for every edge $e
\subset P$. Every linear function $\ell(x) = \inner{c,x}$ yields a (possibly
trivial) orientation on the edges of $P$ and thus determines a relatively open
cone of $\Arr(P)$. It can be shown that for every nonempty face $F \subset P$,
the normal cone $\Ncone{F}{P}$ is partitioned by some relatively open Cones of
$\Arr(P)$. In the language of \cite[Sect.~7.1]{ziegler}, the fan induced by
$\Arr(P)$ \emph{refines} the normal fan of $P$ and the refinement is typically
strict.  It follows that $P$ is a belt polytope if and only if the normal fan of
$P$ coincides with the fan induced by $\Arr(P)$.  The name `belt polytope'
derives from the following fact: two faces $F, F' $ of a belt polytope $P$
satisfy $\LL(F) = \LL(F')$ if and only if $F$ and $F'$ are normally equivalent.
The collection of faces $F$ with fixed $\LL(F)$ are said to be in the same belt. 
Thus, if $P$ is a belt polytope, then $\Lflats(P) = \Lflats(\Arr(P))$ is a
lattice graded by dimension with minimum $\Lbot = \LL(v) = \R^d$ for every
vertex $v$ and maximum $\Ltop = \LL(P)$.

The \Def{Whitney numbers of the second kind} $W_i(\Lflats)$ of a graded poset
$\Lflats$ count the number of elements $a \in \Lflats$ of rank $\rk_\Lflats(a)
= i$.

\begin{prop}\label{prop:zono_ext}
    Let $P$ be a belt polytope of dimension $d$ and let $\Lflats =
    \Lflats(P)$. Then for any cone angle $\CA$
    \[
        \aExt_i(P) \ = \ W_{i}(\Lflats)
    \]
    for all $i=0,\dots,d-1$.
\end{prop}

\begin{proof}
    Let $L \in \Lflats$. From Proposition~\ref{prop:arr_part} we infer that as
    elements of $\CG$
    \begin{equation}\label{eqn:ext_L}
        \sum_{F} [\Ocone{F}{P}]  \ = \ 
         \sum_{F} [L +\Ncone{F}{P}]  \ = \ 
        [\R^d] \, ,
    \end{equation}
    where the sum is over all faces $F \subseteq P$ with $\LL(F) = L$.
    For $i = 0,1,\dots,d-1$ fixed it follows that
    \[
        \sum_{\substack{F \subset P\\ \dim F = i}} [\Ocone{F}{P}] \ = \
        \sum_{\substack{L \in \Lflats\\ \dim L = i}}  \
        \sum_{\substack{F \subset P\\ \LL(F) = L}}
        [\Ocone{F}{P}] \ = \
        \sum_{\substack{L \in \Lflats\\ \dim L = i}} [\R^d] \ = \
        W_{i}(\Lflats) [\R^d] 
    \]
    and applying $\CA$ yields the claim.
\end{proof}

A configuration $z_1,\dots,z_n$ of $n \ge d$ vectors in $\R^d$ is
\Def{generic} if any choice of $d$ vectors are linearly independent. The
proper faces of the associated zonotope $Z$ are parallelepipeds. This
implies that the poset $\Lflats(Z) \setminus \{\Ltop\}$ is isomorphic to the
collection of subsets of $[n]$ of cardinality at most $d-1$ ordered by
inclusion and hence depends only on $n$ and $d$.

\begin{cor}\label{cor:W-rel}
    For $d \ge 1$ and $\CA$ a cone angle
    \[
        \aff \{ \aExt(P)  :  P \subset \R^d \text{ $d$-zonotope} \} \ = \
        \aff \{ \aExt(P)  :  P \subset \R^d \text{ $d$-polytope} \} \ = \
        \{ a \in \R^d : a_0 = 1 \} \, .
    \]
\end{cor}
\begin{proof}
    Proposition~\ref{prop:Ext0} implies that $\subseteq$ holds and
    thus we only need to exhibit $d$ zonotopes whose exterior angle vectors
    are linearly independent. By Proposition~\ref{prop:zono_ext}, it
    suffices to find $d$ zonotopes $Z_0,\dots,Z_{d-1} \subset \R^d$ such
    that the $d$-by-$d$ matrix $A = (a_{ij})_{i,j=0,\dots,d-1}$ with ${a_{ij}
    = W_i(Z_j)}$ has rank $d$. Let $Z_j$ be the zonotope obtained from a
    collection of $d+j$ generic vectors.  Then
    \[
        a_{ij} \ = \ \binom{d + j}{i} \quad \text{ for } i,j = 0,1,\dots,d-1
        \, .
    \]
    Row operations together with Pascal's identity then show that $A$ has
    determinant $1$, which proves the claim.
\end{proof}

The \Def{incidence algebra} $\Incidence(\poset)$ of a finite poset
$(\poset, \preceq)$ is the vector space of all functions
${h : \poset \times \poset \to \C}$, such that $h(a, c) = 0$ whenever
$a \not\preceq c$ and with multiplication
\[
        (g\ast h)(a, c) \  = \ \sum_{a \preceq b \leq c} g(a, b)h(b, c) 
\]
for $g, h \in \Incidence(\poset)$; see Stanley~\cite[Ch.~3]{EC1} for more on
this. The \Def{zeta function} $\zeta_\poset \in \Incidence(\poset)$ is given
by $\zeta_\poset(a,c) = 1$ if $a \preceq c$ and $=0$ otherwise.  The zeta function is
invertible in $\Incidence(\poset)$ with inverse given by the \Def{M\"obius
function} $\mu_\poset = \zeta_{\poset}^{-1}$. More precisely, the M\"obius
function satisfies $\mu_\poset(a,a) = 1$ and 
\[
    \mu_\poset(a,c) \ = \ -\sum_{a \prec b \preceq c} \mu_\poset(b,c)
\]
for $a \prec c$.
For a graded poset $\poset$ of rank $d$, the \Def{characteristic polynomial}
$\chi_\poset(t) \in \Z[t]$ is defined by
\[
    \chi_\poset(t) \ = \ \sum_{a \in \poset} \mu_{\poset}(\Lbot,a) \, t^{d -
    \rk(a)}  \ = \ w_0(\poset) t^d + w_1(\poset) t^{d-1} + \cdots +
    w_d(\poset) \, .
\]
The numbers $w_i(\poset)$, called the \Def{Whitney numbers of the first
kind}, are explicitly given by
\[
    w_i(\poset) \ = \ \sum_{a \,:\, \rk(a) = i} \mu_\poset(\Lbot,a) \, .
\]
\newcommand\CoChar{\psi}%
In particular, $w_0(\poset) = 1$ and $w_d(\poset) = \mu_\poset(\Lbot,\Ltop)$.
The characteristic polynomial $\chi_\Lflats(t)$ where $\Lflats =
\Lflats(\Arr)$ is the lattice of flats of a hyperplane arrangement $\Arr$
captures a number of important properties. For example, Zaslavsky's celebrated
result~\cite{Zaslavsky} states that $|\chi_\Lflats(-1)|$ is the number of
regions of $\Arr$; see also~\cite[Ch.~3.6]{crt}. Here, however, we will be
interested in the characteristic polynomial of the opposite poset
$\Lflats^\op$.

\begin{lem}\label{lem:key}
    Let $P$ be a $d$-dimensional belt polytope with lattice of flats $\Lflats
    = \Lflats(P)$ and let $\CA$ be a cone angle.  For any fixed $L \in
    \Lflats$
    \[
        \sum_{F} \aInt(F,P) \ = \
        (-1)^{d-\dim L}\mu_{\Lflats}(L,\Ltop) \ = \ 
        (-1)^{d-\dim L}\mu_{\Lflats^\op}(\Lbot,L) \, ,
    \]
    where the sum is over all faces $F \subseteq P$ with $\LL(F) = L$.
\end{lem}

The proof makes use of the fact that tangent and normal
Cones are related by polarity.

\begin{prop}\label{prop:NT_polar}
    Let $P \subset \R^d$ be a full-dimensional polytope and $\v \in P$ a
    vertex. Then
    \[
        (\Ncone{\v}{P})^\polar \ = \ \{ \u \in \R^d : \inner{\u,\x} \le 0 \text{
        for all } \x \in \Ncone{\v}{P} \} \ = \ \Tcone{\v}{P} \, .
    \]
\end{prop}
\begin{proof}
    \newcommand\cc{\mathbf{c}}%
    Observe that $\cc \in \Ncone{\v}{P}$ if and only if $\inner{\cc,\v} \ge
    \inner{\cc,\x}$ for all $\x \in P$. That is, if and only if $\inner{\cc,\x -
    \v} \le 0$ for all $\x \in P$ and from $\Tcone{\v}{P} = \cone(-\v + P)$
    we deduce that $\Tcone{\v}{P}^\polar = \Ncone{\v}{P}$.
\end{proof}

\begin{proof}[Proof of Lemma~\ref{lem:key}]
    We again prove the following more general statement over $\CG$
    \begin{equation}\label{eqn:GZ}
        \sum_{F} [\Tcone{F}{P}] \ = \
        (-1)^{d-\dim L} \mu_{\Lflats(P)}(L,\Ltop)  \, [\R^d] \, ,
    \end{equation}
    where the sum is over all faces $F \subseteq P$ with $\LL(F) = L$.

    Let us assume that $L = \Lbot = \{0\}$. Proposition~\ref{prop:NT_polar}
    states that $\Tcone{\v}{P}$ is precisely the polar cone
    $\Ncone{\v}{P}^\polar$.  That is, if $\w \in \interior(\Tcone{\v}{P})$, then
    the hyperplane $\w^\perp$ does not meet $\interior(\Ncone{\v}{P})$. Note
    that since $P$ is a belt polytope, the Cones $\Ncone{\v}{P}$ are the
    regions of $\Arr = \Arr(P)$.

    Hence, for a generic $\w$, the left-hand side of~\eqref{eqn:GZ} is the
    number of regions of $\Arr$ that are not intersected by $\w^\perp$.  By a
    classical result of Greene and Zaslavsky~\cite[Thm.~3.1]{GZ}, this number
    is independent of $\w$ and is exactly $(-1)^{d-\dim
    L}\mu_{\Lflats(P)}(\Lbot,\Ltop)$.

    For $L \neq \Lbot$, let $\pi_L : \R^d \to L^\perp$ be the orthogonal
    projection along $L$. Then $\pi_L(P)$ is a belt polytope and
    $\Lflats(\pi_L(P))$ is isomorphic to the interval $[L,\Ltop] \subseteq
    \Lflats(P)$.
\end{proof}

The following shows that the interior angle vectors of zonotopes are
determined by the Whitney numbers of the first kind. Together with 
Proposition~\ref{prop:zono_ext}, this proves Theorem~\ref{thm:whitney}.

\begin{proof}[Proof of Theorem~\ref{thm:whitney}]
    Let $P$ be a belt polytope with lattice of flats $\Lflats = \Lflats(P)$.
    With the help of Lemma~\ref{lem:key}, we deduce for $L \in \Lflats$
    with $\dim L = i$
    \begin{align*}
      \aInt_i(P)
      & = \ \sum_{\dim F = i} \aInt(F,P)
      \ = \ \sum_{\dim L = i} \sum_{\LL(F) = L}\aInt(F,P)
      \ = \ \sum_{\dim L = i} (-1)^{d-i}\mu_{\Lflats^\op}(\Lbot,L)\\
      & = \ (-1)^{d-i}w_{d-i}(\Lflats^\op). \qedhere
    \end{align*}
\end{proof}

In~\cite{NPS}, Novik, Postnikov, and Sturmfels introduced the
\Def{cocharacteristic polynomial} of the lattice of flats: For a zonotope
$Z$ of dimension $d$ and lattice of flats $\Lflats = \Lflats(Z)$, its
cocharacteristic polynomial is 
\[
    \CoChar_\Lflats(t) \ = \ \sum_{L \in \Lflats} |\mu_\Lflats(L,\Ltop)| \,
    t^{d - \dim L} \ = \ \sum_{i=0}^{d}  |w_{d-i}(\Lflats^\op)| t^{d-i} \ = \
    (-t)^d \chi_{\Lflats^\op}(-\tfrac{1}{t}) \, .
\]
In~\cite{NPS} the coefficients of the cocharacteristic polynomial encoded
invariants of ideals associated to matroids. Here, cocharacteristic
polynomials give us an elegant mean to prove the following theorem, which
proves the uniqueness of~\eqref{eqn:gram} in Theorem~\ref{thm:gram}.

\begin{thm}\label{thm:w-rel}
    For $d \ge 1$, let $\CA : \Cones_d \to \R$ be a cone angle. Then
    \[
        \aff \{ \aInt(Z)  :  Z \subset \R^d \text{ $d$-zonotope} \} =
        \{ (a_0,\dots,a_{d-1}) \in \R^d : a_0 - a_1 + \cdots +
        (-1)^{d-1}a_{d-1} \kern-1pt=\kern-1pt (-1)^{d+1} \}.
    \]
\end{thm}
\begin{proof}
    Using Theorem~\ref{thm:whitney}, it suffices to produce $d$ zonotopes
    $Z_0,\dots,Z_{d-1}$ whose cocharacteristic polynomials are linearly
    independent.

    For $j \ge 0$, let $Z_j$ be the $d$-dimensional zonotope of $d+j$ generic
    vectors and let $\CoChar_{d,j}(t)$ be its associated cocharacteristic
    polynomial. From~\cite[Prop.~4.2]{NPS} we deduce that these polynomials
    satisfy the recurrence
    \[
        \CoChar_{d,j}(t) \ = \ \CoChar_{d-1,j}(t)  + \binom{d-1+j}{j} t
        (t+1)^{d-1} 
    \]
    for $d \ge 1$ and $\CoChar_{0,j}(t) = 1$. We claim that the
    cocharacteristic polynomials $\CoChar_{d,j}(t)$ for ${0 \le j \le d-1}$ are
    \renewcommand\l{\lambda}%
    linearly independent. Indeed, the recursion and the fact that $\deg
    \CoChar_{d,j}(t) = d$ shows that $\sum_j \l_j \CoChar_{d,j}
    = 0$ for $\l_0,\dots,\l_{d-1} \in \R$ 
    if and only if
    \[
    \sum_{j=0}^{d-1} \binom{d-1+j}{j} \l_j \ = \ 0 \quad \text{ and } \quad
    \sum_{j=0}^{d-1} \CoChar_{d-1,j}(t) \l_j \ = \ 0 \, .
    \]
    Iterating this idea, it follows that $\l = (\l_0,\dots,\l_d)$ is in the
    kernel of the $d$-by-$d$ matrix $A$ with entries $\binom{i+j}{j}$ for $i,j
    = 0,\dots,d-1$. Again appealing to Pascal's identity, it is easy to see
    that $\det A = 1$, which completes the proof.
\end{proof}

\section{Connecting angles with M\"obius inversion}\label{sec:incalg}

In this section we take an algebraic approach to the occurrence of the Whitney
numbers of the lattice of flats of a belt polytope in the previous sections.
The \Def{face lattice} of a polytope $P$ is the collection $\Lfaces(P)$ of
faces of $P$ ordered by inclusion.  For a given belt polytope $P$, we define a
certain subalgebra of $\Incidence(\Lfaces(P))$. As it will turn out the map $F
\mapsto \LL(F)$ yields a pair of transformations and Theorem~\ref{thm:whitney}
follows from the fact that the two transformations are adjoint. In
particular, we derive a generalization of a result of Klivans and
Swartz~\cite{KlivansSwartz} regarding spherical intrinsic volumes and Whitney
numbers.

\newcommand{\posetq}{\mathcal{Q}}%
Let $\poset, \posetq$ be two posets. A surjective and order preserving map
$\phi : \poset \to \posetq$ induces a linear transformation $\phi_\ast :
\Incidence(\poset) \to \Incidence(\posetq)$ by
\[
    \phi_\ast h(q,q') \ \defeq \ \ \frac{1}{|\phi^{-1}(q')|} 
    \sum_{ \substack{p \in \phi^{-1}(q)\\ p' \in \phi^{-1}(q')}} h(p,p') 
\]
called the \Def{pushforward} of $h$.  Let $\Incidence_\phi(\poset) \subseteq
\Incidence(\poset)$ be the vector subspace of all elements $h \in
\Incidence(\poset)$ such that for all $q,q' \in \posetq$
\begin{equation}\label{eqn:phi}
   \sum_{p \in \phi^{-1}(q)} h(p, p_1') \ = \ \sum_{p \in \phi^{-1}(q)} h(p,
   p_2')
   \qquad \text{ for all } \ p'_1,p'_2  \in \phi^{-1}(q') \, .
\end{equation}
The neutral element $\delta \in \Incidence(\poset)$ is defined by $\delta(x,y)
= 1$ if $x=y$, and $=0$ otherwise. Clearly, $\delta \in
\Incidence_\phi(\poset)$ and thus $\Incidence_\phi(\poset) \neq \emptyset$.  For
an element $h \in \Incidence_\phi(\poset)$, the pushforward simplifies to
\[
    \phi_\ast h(q,q') \ = \ \sum_{p \in \phi^{-1}(q)} h(p,p')
\]
for any $p' \in \phi^{-1}(q')$. 

\begin{prop}\label{prop:Rf_is_alg}
    $\Incidence_\phi(\poset)$ is a subalgebra of $\Incidence(\poset)$ and
    $\phi_\ast : \Incidence_\phi(\poset) \to \Incidence(\posetq) $ is an
    algebra map. 
\end{prop}
\begin{proof}
    Let $g, h \in \Incidence_\phi(\poset)$. For $q, q' \in \posetq$ 
    and $p' \in \phi^{-1}(q')$ arbitrary we compute
    \begin{align*}
            (\phi_\ast g \ast \phi_\ast h)(q, q') 
            \ &= \
              \sum_{s \in \posetq} \phi_\ast g(q, s) \cdot \phi_\ast h(s, q') 
            \ = \
            \sum_{p \in \phi^{-1}(q)} \sum_{s \in \posetq} \sum_{r \in \phi^{-1}(s)} g(p,
            r) \cdot h(r, p')\\
            \ &= \
            \sum_{p \in \phi^{-1}(q)} \sum_{r \in \poset} g(p, r) \cdot h(r, p')
            \ = \ 
            \sum_{p \in \phi^{-1}(q)} (g \ast h)(p, p') \, .
    \end{align*}
    Since the left-hand side does not depend on the choice of $p'$, we see that $g \ast h
    \in \Incidence_\phi(\poset)$ and therefore $\phi_\ast g \ast \phi_\ast h = \phi_\ast
    (g \ast h)$. Since $\delta \in \Incidence_\phi(\poset)$, this shows that
    $\Incidence_\phi(\poset)$ is a subalgebra.
\end{proof}

For a graded poset $\poset$ of rank $d$, we can define a binary operation
$\ast_k : \Incidence(\poset) \times \Incidence(\poset) \to \Incidence(\poset)$
for $k = 0,\dots,d$ by
\begin{equation}\label{eqn:k-prod}
    (g \ast_k h)(a,c) \ \defeq \ \sum_{b \,:\, \rk(b)=k} g(a,b) \, h(b,c) \, .
\end{equation}
By definition $g \ast h = \sum_k g \ast_k h$. It is noteworthy that
$\ast_k$ and $\ast$ are associative operations, i.e., for $0 \leq k \leq l
\leq d$ and $g,h,m \in \Incidence(\poset)$
\[
  g \ast (h \ast_k m) = (g \ast h) \ast_k m \, , \qquad
  g \ast_k (h \ast m) = (g \ast_k h) \ast m \, , \qquad
  g \ast_k (h \ast_l m) = (g \ast_k h) \ast_l m \, .
\]

The proof of Proposition~\ref{prop:Rf_is_alg} carries over verbatim to prove
the following corollary.

\begin{cor}\label{cor:pushfor_k}
    Let $\poset$ and $\posetq$ be ranked posets. If $\phi : \poset \to
    \posetq$ is a surjective order preserving map that preserves rank, then
    \begin{equation}\label{eqn:pushfor_k}
        \phi_\ast( g \ast_k h) \ = \ \phi_\ast g  \ast_k \phi_\ast h \, .
    \end{equation}
\end{cor}

Let $C \subset \R^d$ be a polyhedral cone and let $\Lfaces_+(C)$ the
collection of nonempty faces of $C$ partially ordered by inclusion. For a
given cone angle $\CA$, we note that the interior and exterior angles
\begin{align*}
  \aInt(F,G) \ &= \ \CA(\Tcone{F}{G})  \ = \ \CA(\Tcone{F}{G} + \affL(G)^\perp)\\
  \aExt(F,G) \ &= \ \CA(\Ocone{F}{G}) \ = \ \CA(\Ncone{F}{G} + \affL(F) )\\
\end{align*}
for faces $F \subseteq G \subseteq C$ are naturally elements of the incidence
algebra $\Incidence(C) \defeq \Incidence(\Lfaces_+(C))$. Furthermore, let us
define:
\begin{align*}
  \aInt'(F,G) \ \defeq \ (-1)^{\dim G - \dim F} \aInt(F,G)\,.
\end{align*}

We call two cone angles $\CA, \CAb$ \Defn{complementary} if for all
polyhedral Cones $C \subseteq \R^d$
\begin{equation}\label{eqn:compl}
    \aInt' \ast \bExt \ = \ \delta_C \, ,
\end{equation}
where $\delta_C$ is the neutral element of $\Incidence(C)$. As a
notable example, the standard cone angle $\SCA$ is self-complementary,
i.e. $\widehat{\SCA}' \ast \widecheck{\SCA} = \delta_C$ for all
polyhedral Cones $C \subseteq \R^d$;
see~\cite{McM-angle}. Complementary angles were studied by
McMullen~\cite{McM-polytopealgebra} under the name of \emph{inverse}
angles. In~\cite{McM-polytopealgebra}, an \Defn{angle functional} is a
collection of normalized and simple valuations $\CA_L$ for Cones in
linear subspaces $L \subseteq \R^d$. The angle of a cone
$C \subseteq \R^d$ is then $\CA_L(C)$ where $L = \affL(C)$.  In this
framework, we define the cone angle $\CA_L$ on a linear subspace
$L \subset \R^d$ by $\CA_L(C) \defeq \CA(C + L^\perp)$ for any cone
$C \subseteq L$. The next lemma is an adaptation
of~\cite[Lemma~46]{McM-polytopealgebra}.

\begin{lem}\label{lem:comp_angle}
    For every cone angle $\CA$ there is a complementary cone angle $\CAb$.
\end{lem}
\begin{proof}
    Lemma~46 in \cite{McM-polytopealgebra} guarantees the existence of
    a complementary angle functional $\beta_L$ for the angle functional $\alpha_L$
    as defined above, such that~\eqref{eqn:compl} is satisfied.  It
    can be shown that $\beta_L$ is of the form
    $\CAb_L(C) = \CAb(C + L^\perp)$ for some cone angle $\CAb$.
\end{proof}

\newcommand\newMin{\bot}%
Let $P \subset \R^d$ be a belt polytope and let $\Lfaces = \Lfaces(P)$ be the
collection of faces of $P$. As before, we can interpret $\aInt(F,G)$ and
$\aExt(F,G)$ as elements in $\Incidence(\Lfaces)$, by extending
$\aInt(\emptyset,G) = 1$ if $\dim G \le 0$ and $=0$ otherwise and
$\aExt(\emptyset,G) = 1$ for all $G$. In particular, $\aInt' \ast \bExt =
\delta_\Lfaces$.

Let $\Lflats_0$ be the set $\Lflats(P) \cup \{ \newMin \}$ partially ordered by
inclusion. Recall that for a non-empty face $F \subseteq P$, $\LL(F) =
\affL(F)^\perp$, where $\affL(F)$ is the linear subspace parallel to $F$.
Setting $\LL(\emptyset) \defeq \newMin$, the map $\LL : \Lfaces(P) \to
\Lflats_0(P)$ given by $F \mapsto \LL(F)$ is a surjective order and rank
preserving map.

\begin{thm}\label{thm:push}
    Let $P$ be a belt polytope and $\Lfaces = \Lfaces(P)$.
    For every cone angle $\CA$, we have $\aInt, \aExt \in
    \Incidence_{\LL}(\Lfaces)$ and 
    \[
       \LL_\ast \aExt \ = \  \zeta_{\Lflats_0} 
            \quad \text{ and } \quad 
       \LL_\ast \aInt' \ = \ \mu_{\Lflats_0} \, .
    \]
\end{thm}
\begin{proof}
    Two faces $G$ and $G'$ of $P$ are normally equivalent if $\LL(G) = \LL(G')$.
    Thus equation~\eqref{eqn:phi} is satisfied and $\aExt$ and $\aInt$ are
    elements of $\Incidence_{\LL}(\Lfaces)$.  Let $G \subseteq P$ be a face and
    $U \in \Lflats$ with $U \supseteq \LL(G)$. Then from~\eqref{eqn:ext_L} in
    the proof of Proposition~\ref{prop:zono_ext} and $\aExt(\emptyset, G) = 1$
    we infer that 
    \[
        \sum_{F \in \Lfaces(P), \atop \LL(F) = U} \aExt(F,G) = 1
    \]
    and hence $\big(\LL_\ast \aExt \big) (U,U') = 1 = \zeta(U, U')$ for all
    $U,U' \in \Lflats$ with $U \supseteq U'$.

    By Lemma~\ref{lem:comp_angle}, there is a cone angle $\beta$ complementary
    to $\alpha$.  Using the fact that $\LL_*$ is an algebra map, we deduce
   \[
       \delta_{\Lflats_0} \ = \ \LL_\ast(\delta_{\Lfaces}) \ = \ \LL_\ast(\aInt' \ast
       \bExt) \ = \ \LL_\ast(\aInt') \ast \LL_\ast(\bExt) \;.
   \]
   Replacing $\aExt$ by $\bExt$ 
   above yields
   $\LL_\ast(\bExt) = \zeta_{\Lflats}$ and thus
   $\LL_\ast(\aInt) = \zeta_{\Lflats}^{-1} = \mu_{\Lflats}$.
\end{proof}

Recall that $\SCA(C) = \frac{\vol(C \cap B_d)}{\vol(B_d)}$ is the standard cone
angle. For a polytope $P \subset \R^d$, the $k$-th \Def{spherical intrinsic
volume} is defined as
\begin{equation}\label{eqn:sp_int_vol}
    \overline{\SCA}_k(P) \ \defeq \ 
    \sum_{v}
    \sum_{v \in F}
    \widehat{\SCA}(v,F)
    \widecheck{\SCA}(F,P) \, ,
\end{equation}
\newcommand\aSph{\overline{\CA}}%
where the sum is over all vertices $v \in P$ and $k$-faces $F \subset P$. For a
given cone angle $\CA$, we denote by $\aSph_k(P)$ the generalization
of~\eqref{eqn:sp_int_vol} to $\CA$.  

The machinery developed in this section yields algebraic proofs of
Theorem~\ref{thm:whitney}.

\begin{cor} \label{cor:alg-KS}
    Let $\CA$ be a cone angle and $P$ a $d$-dimensional belt polytope.  For $k
    = 0,\dots,d-1$ the following hold:
    \begin{enumerate}[\rm (i)]
        \item $\aExt_k(P)  =  W_{k}(\Lflats(P))$;
        \item $\aInt_k(P)  =  |w_{d-k}(\Lflats(P)^\op)|$;
        \item $\aSph_k(P)  = |w_{k}(\Lflats(P))|$.
    \end{enumerate}
\end{cor}

Parts (ii) and (iii) were shown by Klivans and Swartz in
\cite{KlivansSwartz} for the standard cone angle; see also~\cite{AL, Schneider17}.
For the proof, we need the following technical result.

\begin{lem}\label{lem:sum_of_rank}
    Let $\phi : \poset \to \posetq$ be a surjective, order and rank preserving
    map between posets with minimal and maximal elements.  If  $f \in
    \Incidence_\phi(\poset)$, then
    \[
        \big(\zeta_\poset \ast_k f\big)(\Lbot_\poset, \Ltop_\poset) \ = \
       \big(\zeta_\posetq \ast_k (\phi_* f)\big)(\Lbot_\posetq, \Ltop_\posetq)\,.
   \]
\end{lem}
\begin{proof}
    Writing out the definition of $\zeta_\poset \ast_k f$ we obtain
   \begin{align*}
     \big(\zeta \ast_k f\big)(\Lbot_\poset, \Ltop_\poset)
       & \ = \ \sum_{p \in \poset \atop \rk(p) = k} f(p, \Ltop_\poset)
       \ = \ \sum_{q \in \posetq \atop \rk(q) = k} \sum_{p \in \phi^{-1}(q)} f(p, \Ltop_\poset)
       \ = \ \sum_{q \in \posetq \atop \rk(q) = k} \big(\phi_* f\big)(q, \Ltop_\posetq)\\
     & \ = \ \big(\zeta \ast_k \phi_* f\big)(\Lbot_\posetq, \Ltop_\posetq)\,.\qedhere
   \end{align*}
\end{proof}

\begin{proof}[Proof of Corollary~\ref{cor:alg-KS}]
   (i) immediately follows from
   Theorem~\ref{thm:push} and
   Lemma~\ref{lem:sum_of_rank} for $k = i$ and $f = \aExt$. Relation
   (ii) follows in the same fashion with $f = \aInt$, but note that we obtain the
   co-Whitney numbers of the first kind. For (iii), we invoke the
   same lemma for $k = 0$ and $f = \aInt \ast_i \aExt$.
\end{proof}
 
This algebraic perspective on angles is very helpful and will facilitate
proofs and computations in the next sections.

\section{Flag-angle vectors}\label{sec:flag}

In this and the next section we prove Theorems~\ref{thm:flag_rels}
and~\ref{thm:flag_whitney}. Our strategy of proof is as follows. First, we
will show that the interior/exterior flag-angle vectors satisfy the relations
stated in Theorem~\ref{thm:flag_rels}. This is done in
Propositions~\ref{prop:flag_ext_rels} and~\ref{prop:flag_int_rels}.  The
algebraic machinery developed in Section~\ref{sec:incalg} enables us to prove
Theorem~\ref{thm:flag_whitney}. To complete the proof of
Theorem~\ref{thm:flag_rels}, we use this combinatorial interpretation of
flag-angle vectors for belt polytopes. It suffices to show that there are no
linear relations on flag-Whitney numbers of lattices of flats. 
For the flag-Whitney numbers of the second kind, this is done in
Section~\ref{sec:flag-whitney} and, by establishing an algebraic
connection (Theorem~\ref{thm:unipot}) between them, this also
addresses the case of flag-Whitney numbers of the first kind.


The following is the analogue of Proposition~\ref{prop:Ext0}.
\begin{prop}\label{prop:flag_ext_rels}
    Let $P$ be a $d$-dimensional polytope and $S \subseteq [d-1]$. Then
    \[
        \faExt_{S}(P) \ = \ \faExt_{S \cup \{0\}}(P) \, .
    \]
\end{prop}
\begin{proof}
    Let $S = \{s_1,\dots,s_k\}$ and set $s_0 \defeq 0$.  Unravelling the
    definition of exterior flag-angle vectors (see~\eqref{eqn:flag}), we compute
    \begin{align*}
        \faExt_{S \cup \{0\}}(P) \ &= \ 
     \sum_{F_0 \subset F_1 \subset F_2 \subset \cdots \subset F_k}
    \aExt(F_0,F_1) \,  \aExt(F_1,F_2) \cdots \aExt(F_k,P) \\
    \ &= \
     \sum_{F_1 \subset F_2 \subset \cdots \subset F_k}
     \,  \aExt(F_1,F_2) \cdots \aExt(F_k,P) 
    \sum_{F_0 \subset F_1} \aExt(F_0,F_1)\\
    \ &= \
     \sum_{F_1 \subset F_2 \subset \cdots \subset F_k}
     \,  \aExt(F_1,F_2) \cdots \aExt(F_k,P) \\
     &= \ \faExt_{S}(P) \, ,
    \end{align*}
    where the sums are over faces $F_i$ with $\dim F_i = s_i$ for
    $i=0,\dots,k$ and where the third equality follows from
    Proposition~\ref{prop:Ext0}.
\end{proof}

As for the linear relations on \emph{interior} flag-angle vectors, we
take a more algebraic approach.  Let $P$ be a $d$-dimensional polytope
with face lattice $\Lfaces = \Lfaces(P)$ and
${S = \{s_1 < s_2 < \cdots < s_k\}}$ with $S \subseteq [0,
d-1]$. Using~\eqref{eqn:k-prod} together with the fact that
$\rk_\Lfaces(F) = \dim F - 1$, we can give the following expression
for the $S$-entry of the interior flag-angle vector
\[
    \faInt_{S}(P) \ = \ 
    (\zeta_{\Lfaces}  \ast_{s_1+1} 
    \aInt \ast_{s_2+1} 
    \cdots \ast_{s_k+1} \aInt)(\emptyset,P)   \, ,
\]
where the operation $\ast_k$ was introduced in 
\eqref{eqn:k-prod}.

\begin{prop}\label{prop:flag_int_rels}
    Let $P$ be a $d$-polytope. For $S = \{0 \le s_1 < s_2 <
    \cdots < s_k \le d-1 \}$ set $t \defeq \min(S \cup \{d\})$. Then 
    \[
        \sum_{i=0}^{t-1} (-1)^{i} \aInt_{S \cup \{i\}}(P) \ = \
        (-1)^{t+1} \aInt_S(P) \, .
    \]
\end{prop}

\begin{proof}
    Recall that the M\"obius function $\mu_\Lfaces = \zeta_{\Lfaces}^{-1}$ is
    given by $\mu_{\Lfaces}(F,G) = (-1)^{\dim G - \dim F}$ for faces $F
    \subseteq G \subset P$. For a fixed face $G$, Proposition~\ref{prop:Int0}
    yields
    \[
        (\mu_\Lfaces \ast \aInt)(\emptyset,G) \ = \ - \sum_{F} (-1)^{\dim F}
        \aInt(F,G) \ = \ - \sum_{i=0}^{\dim G-1} (-1)^i \aInt_i(G) +
        (-1)^{\dim G+1} \ = \ 0 \, .
    \]
    The result now follows by evaluating 
    \[
        \mu_{\Lfaces} \ast \aInt  \ast_{s_1+1} \aInt \ast_{s_2+1} \cdots
        \ast_{s_k+1} \aInt
    \]
    at $(\emptyset,P)$.
\end{proof}

Let $\poset$ be a graded poset of rank $d+1$ and let $S = \{s_1 < s_2 < \dots <
s_k\} \subseteq [d]$. The \Def{flag-Whitney numbers of the second kind} as defined
in the introduction are given by
\[
    W_S(\poset) \ = \ ( \zeta_\poset \ast_{s_1} \zeta_\poset \ast_{s_2} \cdots
    \ast_{s_k} \zeta_\poset) (\Lbot,\Ltop) \, .
\]
Similarly, we define the \Defn{flag-Whitney numbers of the first kind} by
\[
    w_S(\poset) \ \defeq \ 
     ( \mu_\poset \ast_{s_1} \mu_\poset \ast_{s_2} \cdots
    \ast_{s_k} \zeta_\poset) (\Lbot,\Ltop) \ = \
    \sum \mu(\Lbot, c_1) \mu(c_1, c_2) \cdots \mu(c_{k-1}, c_{k})\; ,
\]
where the sum is over all chains $\Lbot \prec c_1 \prec c_2 \prec \cdots \prec c_k$ with $\rk c_i = s_i$ for $i=1,\dots,k$.
Now the same reasoning as in the proof of Corollary~\ref{cor:alg-KS} yields
Theorem~\ref{thm:flag_whitney}:

\begin{proof}[Proof of Theorem~\ref{thm:flag_whitney}]
    Let $\Lfaces = \Lfaces(P)$ be the face lattice of $P$ and let $\Lflats_0$ be
    the poset $\Lflats(P)$ with a new minimal element $\Lbot_{\Lflats_0}$
    adjoined. The maximal element of $\Lflats_0$ is $\Ltop_{\Lflats_0} =
    \LL(P)$.  We also set $t_i \defeq s_i + 1$ and $f \defeq \aExt \ast_{t_2}
    \dots \ast_{t_k} \aExt$. Note that $f \in \Incidence_\LL(\Lfaces)$ and with
    Lemma~\ref{lem:sum_of_rank} we compute
   \begin{align*}
     \aExt_S(P)
        &\ = \ \big(\zeta_\Lfaces \ast_{t_1} \aExt \ast_{t_2} \dots \ast_{t_k}
        \aExt\big)(\emptyset, P) \ = \  \big(\zeta_\Lfaces \ast_{t_1}
        f\big)(\Lbot_{\Lfaces}, \Ltop_{\Lfaces}) \ = \ \big(\zeta_\Lfaces
        \ast_{t_1} (\LL_\ast f)\big)(\Lbot_{\Lflats_0}, \Ltop_{\Lflats_0})\\
        &\stackrel{\eqref{eqn:pushfor_k}}{\ = \ }  \big(\zeta_{\Lflats_0}
        \ast_{t_1} \zeta_{\Lflats_0} \ast_{t_2} \dots \ast_{t_k}
        \zeta_{\Lflats_0}\big)(\Lbot_{\Lflats_0}, \Ltop_{\Lflats_0}) =
        \big(\zeta_{\Lflats} \ast_{s_1} \zeta_{\Lflats} \ast_{s_2} \dots
        \ast_{s_k} \zeta_{\Lflats}\big) (\Lbot_{\Lflats}, \Ltop_{\Lflats})\\
        & \ = \ W_{S}(\Lflats) \ = \ W_{d-S}(\Lflats^\op) \,.
   \end{align*}
   Similarly for the second statement, for $g \defeq \aInt \ast_{t_2} \dots
   \ast_{t_k} \aInt \in \Incidence_\LL(\Lfaces)$ we obtain:
   \begin{align*}
     \aInt_S(P)
     &\ = \ \big(\zeta_\Lfaces \ast_{t_1} \aInt \ast_{t_2} \dots \ast_{t_k}
       \aInt\big)(\emptyset, P) = \big(\zeta_\Lfaces \ast_{t_1}
       f\big)(\Lbot_{\Lfaces}, \Ltop_{\Lfaces})
     \ =  \big(\zeta_\Lfaces \ast_{t_1} (\LL_\ast f)\big)(\Lbot_{\Lflats_0},
       \Ltop_{\Lflats_0})\\
       &\stackrel{\eqref{eqn:pushfor_k}}{\ = \ } (-1)^{d + 1 - t_1} \cdot \big(\zeta_{\Lflats_0} \ast_{t_1} \mu_{\Lflats_0} \ast_{t_2} \dots \ast_{t_k} \mu_{\Lflats_0}\big)(\Lbot_{\Lflats_0}, \Ltop_{\Lflats_0})\\
     &\ = \ (-1)^{d - s_1} \cdot \big(\zeta_{\Lflats} \ast_{s_1} \mu_{\Lflats} \ast_{s_2} \dots \ast_{s_k} \mu_{\Lflats}\big) (\Lbot_{\Lflats}, \Ltop_{\Lflats})
      \ = \ (-1)^{d-s_1} \cdot w_{d-S}(\Lflats(P)^\op)\,.
      \qedhere
   \end{align*}
\end{proof}

\newcommand\Poly{\mathcal{P}}%
\newcommand\dbrackets[1]{[\![#1]\!]}%
In order to complete the proof of Theorem~\ref{thm:flag_rels}, we observe that
the flag-Whitney numbers of the second kind determine the flag-Whitney numbers
of the first kind. We show this in more generality. Let $\poset$ be a finite
poset with $\Lbot$ and $\Ltop$ and let $R \defeq \C\dbrackets{z_a : a \in
\poset}$ be the ring of formal power series with variables indexed by elements
of $\poset$.  For a unipotent $g \in \Incidence(\poset)$, i.e., $g(a,a) = 1$
for all $a \in \poset$, we define
\[
    F_g(\z) \ \defeq \ \sum
    g(\Lbot, c_1) \, z_{c_1} \,
    g(c_1, c_2) \, z_{c_2} \,
    \cdots
    \, z_{c_{k-1}} \, g(c_{k-1}, c_k) \, z_{c_k} \, ,
\]
where the sum is over all multichains $\Lbot \prec c_1 \preceq c_2 \preceq
\cdots \preceq c_k \prec \Ltop$. Since every multichain comes from a unique
chain, we can rewrite $F_g(\z)$ to
\[
    F_g(\z) \ = \
    \sum_{\Lbot \prec b_1 \prec b_2 \prec \cdots \prec b_k \prec \Ltop}
        g(\Lbot, b_1) \frac{z_{b_1}}{1 - z_{b_1}}
        g(b_1, b_2) \frac{z_{b_2}}{1 - z_{b_2}}
        \cdots
         g(b_{k-1}, b_k)  
        \frac{z_{b_k}}{1 - z_{b_k}}    \, .
\]

If $\poset$ is a graded poset of rank $d+1$, then for $g = \zeta$, we get
\[
    G_\poset(\q) \ \defeq \ F_\zeta(z_a = q_{\rk(a)} : a \in \poset) \ = \
    \sum_{S \subseteq [d]} W_S(\poset) \prod_{i \in S} \frac{q_i}{1-q_i} \ \in
    \ \C\dbrackets{q_1,\dots,q_d} \, .
\]
Since the elements $\frac{q_i}{1-q_i}$ for $i=1,\dots,d$ are algebraically
independent over $\C\dbrackets{q_1,\dots,q_d}$, $G_\poset(\q)$ encodes the flag-vector of $\poset$.  The
relation to the flag-Whitney numbers of the second kind follows from the next
theorem.

\begin{thm}\label{thm:unipot}
    Let $g \in \Incidence(\poset)$ be  unipotent.  Then
    \[
        F_g(\tfrac{1}{\z}) \ = \ F_{g^{-1}}(\z) \, .
    \]
\end{thm}
\begin{proof}
    We observe that
    \[
        F_g(\tfrac{1}{\z}) \ = \
        \sum_{\Lbot \prec b_1 \prec b_2 \prec \cdots \prec b_k \prec \Ltop}
        g(\Lbot, b_1) \frac{-1}{1 - z_{b_1}}
        g(b_1, b_2) \frac{-1}{1 - z_{b_2}}
        \cdots
        g(b_{k-1}, b_k)
        \frac{-1}{1 - z_{b_k}}
        \, .
    \]
    The coefficient $g(\Lbot, b_1) g(b_1, b_2) \cdots g(b_{k-1}, b_k)$ now
    contributes to every multichain supported on a subset of $\{ b_1, b_2
    , \dots, b_k\}$. Rewriting, this is the same as
    \[
        F_g(\tfrac{1}{\z}) \ = \
        \sum_{\Lbot \prec a_1 \prec a_2 \prec \cdots \prec a_l \prec \Ltop}
        h(\Lbot, a_1) \frac{z_{a_1}}{1 - z_{a_1}}
        h(a_1, a_2) \frac{z_{a_2}}{1 - z_{a_2}}
            \cdots
            h(a_{l-1}, a_l)
            \frac{z_{a_l}}{1 - z_{a_l}} \, ,
    \]
    where for $u \prec v$
    \begin{align*}
        h(u,v) \ \defeq& \
        \sum_{u \prec b_1 \prec b_2 \prec \cdots \prec b_k \prec v}
        (-1)^k
        g(u,b_1) g(b_1,b_2) \cdots g(b_k,v)\\
        \ =& \ \sum_{k \ge 0} (-1)^k (g-\delta)^k(u,v) \ = \ g^{-1}(u,v) \, .
        \qedhere
    \end{align*}
\end{proof}

The above computation is reminiscent of calculation of the antipode applied to
the quasisymmetric function associated to a graded poset in
Ehrenborg~\cite{Ehrenborg-hopf}. Applying this statement to a pair $\alpha,
\beta$ of complementary angles allows us to directly relate interior and
exterior flag-angles:

\begin{cor}\label{cor:int_to_ext}
    Let $\CA$ be a cone angle with complementary cone angle $\CAb$.  For every
    $d$-polytope $P$, the interior and exterior flag angle vectors are related
    via
    \[
        \sum_{S} (-1)^{d - t} \aInt_S(P) \prod_{i \in S} x_i
        \ = \ \sum_{S} \bExt_S(\poset) \prod_{i \in S} -(x_i+1)
        \ \in \ \C[x_1,\dots,x_d] \, ,
    \]
    where the sums are over all $S \subseteq [0, d-1]$ and
    $t = \min(S \cup \{d\})$.
\end{cor}
\begin{proof}
    Let $\Lfaces = \Lfaces(P)$ and
    $x_i \defeq \frac{q_i}{1 - q_i} \in R$. Then
    $q_i = \frac{x_i}{x_i + 1}$ and
    $\frac{-1}{1 - q_i} = -x_i-1$. Using Theorem~\ref{thm:unipot}, we compute
    \begin{align*}
      \sum_{S} (-1)^{d - t} \aInt_S(P) \prod_{i \in S} x_i 
      \ &= \ \sum_{S} \aInt'_S(P) \prod_{i \in S} x_i
      \  = \ \sum_{S} \aInt'_S(P) \prod_{i \in S} \frac{q_i}{1-q_i}\\
      \ &= \ F_{\aInt'}(z_a = q_{\rk(a)} : a \in \Lfaces)
      \  = \ F_{\bExt}(z_a = q_{\rk(a)}^{-1} : a \in \Lfaces)\\
      \ &= \ \sum_{S} \bExt_S(P) \prod_{i \in S} \frac{-1}{1-q_i}
      \  = \ \sum_{S} \bExt_S(P) \prod_{i \in S} -(x_i+1)\,,
    \end{align*}
    where each sum ranges over all $S \subseteq [0, d-1]$.
\end{proof}

\begin{proof}[Proof of Theorem~\ref{thm:flag_rels}]
    Propositions~\ref{prop:flag_ext_rels} and~\ref{prop:flag_int_rels} yield
    that the linear relations given in Theorem~\ref{thm:flag_rels} hold. In
    particular, this shows that the dimensions of the affine hulls of
    interior/exterior flag-angles is at most $2^{d-1}-1$.

    From Theorem~\ref{thm:flag_whitney}, we infer that
    \[
        \aff \{\faExt(P) : \text{$P$ $d$-polytope} \} \ \supseteq \
        \aff \{\faExt(Z) : \text{$Z$ $d$-zonotope} \} \ = \ 
        \aff \{\FW(\Lflats(Z)^\op) : \text{$Z$ $d$-zonotope} \} \, .
    \]
    Theorem~\ref{thm:zono_span}, that we will prove in the next section, shows
    that the dimension of the affine hull of flag-vectors of $\Lflats(Z)$
    where $Z$ ranges over all $d$-dimensional zonotopes is of dimension
    $2^{d-1}-1$. This proves the claim for exterior flag-angle vectors.

    The same reasoning applies to the interior flag-angle vectors and it
    suffices to determine the affine span of $(w_S(\Lflats(Z)^\op))_S$ for
    $d$-dimensional zonotopes $Z$.  Analogously to
    Corollary~\ref{cor:int_to_ext}, Theorem~\ref{thm:unipot} implies that the
    spaces of flag-Whitney numbers of the first and of the second kind spanned
    by posets of rank $d+1$ are linearly isomorphic, which completes the
    proof.
\end{proof}

\section{Flag-Whitney numbers and zonotopes}\label{sec:flag-whitney}

Let $\poset$ be a graded poset with $\Lbot$ and $\Ltop$ of rank $d+1$.  It was
shown by Billera and Hetyei~\cite{BilleraHetyei} that flag-vectors of general
graded posets do not satisfy any nontrivial linear relation. That is
\[
    \dim \aff \{ \FW(\poset) \in \R^{2^d} \ : \ \poset \text{ graded poset of
    rank $d+1$} \} \ = \ 2^{d} - 1 \, .
\]
The only linear relation is given by $W_\emptyset (\poset)= 1$.

In light of Theorem~\ref{thm:flag_whitney}, we can complete the proof of
Theorem~\ref{thm:flag_rels} for exterior flag-angles by proving the
following refinement of the result of Billera and Hetyei.

\begin{thm}\label{thm:zono_span}
    The flag-vectors of lattices of flats of $(d+1)$-dimensional zonotopes
    span the flag-vectors of rank $d+1$ posets. That is,
    \[
        \dim\aff\{\FW(\Lflats(Z)) \ : \ Z \text{ zonotope of dimension $d+1$}\}
        \ = \ 2^d - 1 \, .
    \]
\end{thm}

The result is analogous to that of Billera--Ehrenborg--Readdy~\cite{BER}
where it is shown that the flag-vectors of face lattices of zonotopes span the
space of flag-vectors of Eulerian posets. For the proof of
Theorem~\ref{thm:zono_span}, we will employ the coalgebra techniques developed
in~\cite{BER}.

\newcommand{\A}{\mathcal{A}} 
Let $\A = k\langle a,b \rangle$ be the polynomial ring in noncommuting
variables $a$ and $b$. This is a graded algebra $\A = \bigoplus_{d\geq 0}
\A_d$ and a basis for $\A_d$ is given by $\{a-b,b\}^d$. The \Def{ab-Index} of
a graded poset $\poset$ of rank $d+1$ is given by
\[
    \Psi(\poset) \ = \ \sum_{S \subseteq [d]} W_S(\poset) x(S)
\]
where $x(S) = x_1x_2\dots x_d \in \{a-b,b\}^d$ with $x_i = b$ if and only if $i \in S$.

Following~\cite{BER}, we consider two natural operations on zonotopes: If
$Z \subset \R^d$ is a zonotope, then $\abE(Z) \defeq Z \times [0,1] \subset
\R^{d+1}$ is a zonotope of dimension $\dim Z + 1$. This is clearly a
combinatorial construction and for the lattice of flats $\Lflats = \Lflats(Z)$,
we note that
\[
    \abE(\Lflats) \ \defeq \ \Lflats(\abE(Z)) \ = \ \Lflats \times C_1 \, ,
\]
where $C_1 = \{ \Lbot \prec \Ltop \}$ is the chain on $2$ elements.
A vector $u \in \R^d$ is in general position with respect to $Z$ is $u$ is not
parallel to any face of $Z$. It can be shown
(see~\cite[Ch.~7]{White-matroids}) that the lattice of flats of $\abM(Z) \defeq
 Z + [0,u]$ is independent of the choice of $u$ and given by
\[
    \abM(\Lflats) \ \defeq \ \Lflats(\abM(Z)) \ = \ \abE(\Lflats) \setminus
    \{ x \in \abE(\Lflats) : \rk(x) = d+1 \} \, .
\]
Let $\abP(Z)$ be the orthogonal projection of $\abM(Z)$ onto the hyperplane
$u^\perp$.  This again is a combinatorial operation and $\abP(\Lflats) \defeq
\Lflats(\abP(\abM(Z)))$ is obtained from $\Lflats$ by deleting the coatoms, that
is, elements of $\rk(\Lflats) - 1$.

In order to determine the effect on the ab-index, we introduce derivations
$\abR, \abR': \A \to \A$ defined on the variables by $\abR(a) \defeq \abR(b)
\defeq ab$ and $\abR'(a) \defeq R'(b) \defeq ba$ and linearly extended via
\[
    \abR(xy)  \  \defeq \ \abR(x)y + x\abR(y)  \qquad
    \abR'(xy)  \ \defeq \ \abR'(x)y + x\abR'(y)
\]
for monomials $x,y$. Note that both derivations are
homogeneous and map $\A_d$ into $\A_{d+1}$. We also define linear maps
$\abE,\abM,\abP:
\A \to \A$ on monomials $x$ by
\[
    \begin{array}{l@{\ \ \defeq \ \ }l@{\qquad}l@{\ \ \defeq \ \ }l}
        \abP(xa) & x & \abE(x) & xa + bx + \abR(x) \\
        \abP(xb) & 0 & \abM(x) & \abP(\abE(x)) \, . \\
    \end{array}
\]
In particular, we have
\[
    \abM(xa) \ = \ xa + bx + \abR(x) \ = \ \abE(x) \qquad \text{ and } \qquad
    \abM(xb) \ = \
    xb \, .
\]

The following result can be easily obtained by inspecting chains.
\begin{lem}
    Let $Z$ be zonotope and $\Lflats = \Lflats(Z)$ its lattice of flats. Then
    \begin{align*}
        \Psi(\abE(\Lflats))    \ &= \ \abE(\Psi(\Lflats)) \ = \
        \Psi(\Lflats)b + a\Psi(\Lflats) + \abR'(\Psi(\Lflats)),\\
        \Psi(\abP(\Lflats))  \ &= \ \abP(\Psi(\Lflats)), \text{ and}\\
        \Psi(\abM(\Lflats))  \ &= \ \abP(\abE(\Psi(\Lflats))) \, .
    \end{align*}
\end{lem}

\begin{proof}[Proof of Theorem~\ref{thm:zono_span}]
\newcommand{\ZZ}{\mathcal{Z}}%
    For $d \ge 0$ let
    \[
        \ZZ_d \ \defeq \
        \mathrm{span} \{ \Psi(\Lflats(Z)) : Z \text{ zonotope of dimension $d+1$} \} \
        \subseteq \ \A_d \, .
    \]
    We show by induction on $d$ that $\ZZ_d = \A_d$. For $d=1$ this is clearly
    true.  Assume that $\ZZ_d = \A_d$. The key observation is that if $x$ is
    any monomial in $\A_d = \ZZ_d$, then also $\abM(x) \in \ZZ_{d}$ and
    $\abE(x) \in \ZZ_{d+1}$.

    \begin{enumerate}[\rm (i)]
        \item $xba \in \ZZ_{d+1}$ for all $x \in \ZZ_{d-1}$:
        \[
            2 \abE(xb) - \abM(\abE(xb)) \ = \ xba  \, .
        \]

        \item $xab \in \ZZ_{d+1}$ for all $x \in \ZZ_{d-1}$:
    \[
       \abM(\abE(xa)) - \abE(xa + bx + \abR(x)) \ = \ xab \, .
    \]

    \item $xba^n \in \ZZ_{d+1}$ for all $x \in \ZZ_{d - n}, n = 1,
     \dots, d - 1$:\\
    For $n = 1$ this is just (i). We may assume that the claim holds for all
    values $< n+1$ and compute
    \[
      \abE(xba^n)  \ = \ xba^{n+1} + xba^nb + \sum_{i=1}^{n}x_iba^{i}
    \]
    for some $x_i \in A_{d-i-1}$. Since $x_iba^i \in \ZZ_{d+1}$ by induction
    and $xba^nb \in \ZZ_{d+1}$ by (ii), we see that $xba^{n+1} \in \ZZ_{d+1}$.

    \item $xab^n \in \ZZ_{d+1}$ for all $x \in \ZZ_{d - n}, n = 1, \dots, d -
        1$:\\
    For $n = 1$ this is just (ii). Assume the statement holds for all values
    $<n+1$:
    \[
      \abE(xab^n) \ = \ xab^{n+1} + xab^na + \sum_{i=1}^{n}x_iab^{i}
    \]
    for some $x_i \in A_{d-i-1}$. Since $x_iab^i \in \ZZ_{d+1}$ by induction
    and $xab^na \in \ZZ_{d+1}$ by (i), we see that $xab^{n+1} \in \ZZ_{d+1}$.
    \end{enumerate}

    Since every monomial in $\A_{d+1}$ which contains at least one $a$ and $b$
    is of either the form $xab^n$ or $xba^n$, we see that it remains to show
    that $a^{d+1}$ and $b^{d+1}$ are in $\ZZ_{d+1}$ as well. For that we
    compute
    \begin{align*}
      \abE(a^d) &\ = \ a^{d+1} + ba^d + \abR(a^d)\\
      \abE(b^d) &\ = \ b^da + b^{d+1} + \abR(b^d)
    \end{align*}
    Since  $ba^d, b^da, \abR(a^d), \abR(b^d) \in \ZZ_{d+1}$, this finishes the
    proof.
\end{proof}

In fact we have proven the following statement:
\begin{cor}
    For $d \ge 0$, a vector space basis of $\A_d$ is given by
    \[
        \bigl\{ \Phi(\Lflats( \sigma [0,1] )) : \sigma \in
        \{\abE,\abM\circ\abE\}^d \bigr\} \, .
    \]
\end{cor}
\begin{proof}
  In the proof of Theorem \ref{thm:zono_span} we only needed elements of the
  form $E(x)$ and $M(E(x))$, $x \in \A_d$, to span $\A_{d+1}$, thus the
  assertion follows by induction.
\end{proof}

\section{Spherical intrinsic volumes and Grassmann angles}
\label{sec:spherical_intrinsic_volumes}
We have already seen spherical intrinsic volumes in Section~\ref{sec:incalg},
but as a further introduction and motivation to flag-angles, we will give a
slightly more general account here. Moreover, we will see how interior and
exterior flag-angles are connected. A second goal of this chapter is to show how
we can adapt and generalize many results by Grünbaum~\cite{Grunbaum-grassmann}
to a more general version of Grassmann angles.

For convenience, we define for a cone $C \subset \R^d$ the \Def{completion} as
$\cpl C \defeq C + \affL(C)^\perp$. Recall that $\SCA$ denotes the spherical
volume. For $0 \leq r \leq d$, the \Defn{$r$-th spherical intrinsic volume}
$\SCA^r$ of a cone $C \in \Cones^d$ is defined as:
\[
    \SCA^r(C) \ \defeq \ \sum_{F \subseteq C \text{ face} \atop \dim F = r}
    \SCA(\cpl F) \cdot \SCA( \Ocone{F}{C} ) \, ,
\]
where $\Ocone{F}{C} = \cpl \Ncone{F}{C}  = \Ncone{F}{C} + \affL(F)$.  Note that
$\SCA^d(C) = \SCA(C)$ and $\SCA^i(C) = \SCA^{d-i}(C^\polar)$ for all $i$. The name
\emph{spherical} intrinsic volume stems from the similarity to a formula for the
usual intrinsic volumes $V^r$ of a polytope $P$
\[
   V^r(P) = \sum_{\dim F = r} \vol_r(F) \cdot \SCA(\Ocone{F}{P})\,,
\]
where $\vol_r$ denotes the usual $r$-dimensional volume; see~\cite[Section
4.2]{Schneider}. The intersection of a face $F \subseteq C$ with the unit sphere
$\Sphere^d$ is a spherical polytope and $\SCA(\cpl F)$ is the normalized
spherical volume.

We would like to replace $\SCA$ in the definition of $\SCA^r$ with more general
valuations. In fact, we could replace both occurrences of $\SCA$ with different
valuations. Let $\alpha$, $\beta$ be cone angles. For $0 \leq r \leq d$, we
define the \Defn{generalized spherical intrinsic volume} $\xi^r = \xi^r(\alpha,
\beta) : \Cones^d \to \R$ by
\begin{equation}
    \xi^r(C) \ \defeq \ 
    \sum_{F \subseteq C \text{ face} \atop \dim F = r} \alpha(\cpl F) \cdot
    \beta(\Ocone{F}{C})\,,
    \label{eqn:gen_sph_int_vol}
\end{equation}
It is well known that $\SCA^r$ is a valuation for all $r$;
see~\cite[Lemma~2.3.2]{Schneider_ConvexCones} for an elementary proof. This
remains true for the generalized spherical intrinsic volumes:

\begin{thm}\label{thm:generalized_spherical_intrinsic_volume_is_valuation}
    Let $\alpha, \beta : \Cones^d \to \R$ be cone angles. Then $\xi^r(\alpha,
    \beta)$ is a valuation for $0 \leq r \leq d$.
\end{thm}

Note that  $\xi(\alpha, \beta)$ is not a simple valuation: if $C$ is a linear
subspace of dimension $r < d$, then $\xi^r(\alpha,\beta)(C) =
\alpha(\R^d)\beta(\R^d) = 1$. Furthermore, recall that we can view the
associated interior and exterior cone angles $\aInt$ and $\bExt$ as elements in
the incidence algebra of the face lattice $\Lfaces(C)$. This allows the
interpretation
\[
    \xi^r(C) \ = \ \big(\aInt \ast_r \bExt\big)(\lineal C, C) \, .
\]
Theorem~\ref{thm:generalized_spherical_intrinsic_volume_is_valuation} suggests
that higher products such as $(\aInt_1 \ast_r \aExt_2) \ast_s \aExt_3$ are
valuations as well. This, unfortunately, is not the case as the simplicity of
$\alpha$ and $\beta$ is essential in the proof of
Theorem~\ref{thm:generalized_spherical_intrinsic_volume_is_valuation}.

\begin{proof}[Proof of
Theorem~\ref{thm:generalized_spherical_intrinsic_volume_is_valuation}]
    By~\cite{Sallee} it suffices to show that $\xi = \xi^r(\alpha, \beta)$ is a
    \emph{weak} valuation: For every cone $C \subset \R^d$ and $H$ a linear
    hyperplane we need to show that 
    \begin{equation}\label{eqn:xi_valuation}
        \xi(C)  \ = \ \xi(C \cap H^\le) + \xi(C \cap H^\ge) - \xi(C \cap H)  \,
        .
    \end{equation}
    It is sufficient to assume that $C \not\subseteq H$ and that $H$ meets the
    relative interior of $C$. Then the cones $C^\le \defeq C \cap H^\le$ and
    $C^\ge \defeq C \cap H^\ge$ are of the same dimension as $C$ and $C^=
    \defeq C \cap H$ is of dimension $\dim(C) - 1$.

    To show~\eqref{eqn:xi_valuation}, we need to consider all $r$-faces of
    $C^\leq$, $C^\geq$, and $C^=$. These faces are either faces of $C$ or are
    obtained by intersecting faces of $C$ with $H^\leq$, $H^\geq$ or $H$ in
    the following ways. Let $F$ be an $r$-face of $C$.
    \begin{enumerate}[{Case} 1.]
        \item If $\relint(F) \cap H = \emptyset$, then $F$ is contained in
            $H^\leq$ or $H^\geq$ and $F$ is an $r$-face of $C^\leq$ or
            $C^\geq$, respectively.
        \item If $F \subseteq H$, then $F$ is an $r$-face of $C^\leq$,
            $C^\geq$, and $C^=$.
        \item If $H$ intersects $\relint(F)$ in a proper subset, then $F \cap
            H^\leq$ is an $r$-face of $C^\leq$ and $F \cap H^\geq$ is an
            $r$-face of $C^\geq$. Further more $F \cap H$ is an $(r-1)$-face
            of $C^\leq$, $C^\geq$ and $C^=$.
    \end{enumerate}
    Furthermore, let $G \subseteq C$ be an $(r+1)$-face.
    \begin{enumerate}[{Case} 4.]
        \item If $\relint(G) \cap H \neq \emptyset$, then $G \cap H$ is an
            $r$-face of $C^\leq$, $C^\geq$, and $C^=$.
    \end{enumerate}

    We consider the contributions of each case to \eqref{eqn:xi_valuation}
    separately:

    Case 1. Without loss of generality we can assume that
    $F \subseteq H^\le$ so that $F$ is a face of $C^\le$ as well. In
    this case $\Ocone{F}{C} = \Ocone{F}{C^\le}$ and thus $F$ gives the
    same contribution to $\xi(C)$ and $\xi(C^\le)$ and none to
    $\xi(C^\ge)$ and $\xi(C^=)$.

    Case 2. As $F$ is a face of all four cones, we have
    $\Ocone{F}{C} = \Ocone{F}{C^{\le}} \cup \Ocone{F}{C^{\ge}}$ with
    $\Ocone{F}{C^{\le}} \cap \Ocone{F}{C^{\ge}} = \Ocone{F}{C^=}$. The
    contribution on the right-hand side is then
    \[
        \alpha(\cpl F)( \beta(\Ocone{F}{C^\le}) + 
        \beta(\Ocone{F}{C^\ge}) - 
        \beta(\Ocone{F}{C^=})) \, ,
    \]
    which equals to $\alpha(\cpl F)\beta(\Ocone{F}{C})$ as $\beta$ is a
    valuation.

    Case 3. Set $F^\le = F \cap H^\le$ and $F^\ge = F \cap H^\ge$,
    which are faces of $C^\le$ and $C^\ge$, respectively. Since the
    normal cone is polar to the tangent cone and the tangent cone of a
    face $F$ is determined by any neighborhood of a point
    $q \in \relint(F)$, we have
    $\Ocone{F}{C} = \Ocone{F^\le}{C^\le} = \Ocone{F^\ge}{C^\ge}$. The
    contribution on the right-hand side is therefore
    \[
        (\alpha(\cpl F^\le) + \alpha(\cpl F^\ge)) \beta(\Ocone{F}{C})
    \]
    which is precisely $\alpha(\cpl F) \beta(\Ocone{F}{C})$ since $\alpha$ is a
    simple valuation.

    Case 4. Here $F^= = F \cap H$ is a common face of $C^\le$,
    $C^\ge$, and $C^=$. Since
    $\Ocone{F^=}{C^=} = \Ocone{F^=}{C^{\le}} \cup
    \Ocone{F^=}{C^{\ge}}$ and $\beta$ is a simple valuation, the
    contribution to the right-hand side is $0$.
\end{proof}

Since $\xi^r = \xi^r(\alpha, \beta)$ is a valuation, we immediately
obtain the following from the Brianchon-Gram relation. For a $d$-polytope
$P$ and a face $F \subseteq P$ define
$\widehat{\xi}^r_i(P) \defeq \sum_{F} \xi^r(T_FP)$ where the
sum is over all $i$-faces $F$ of $P$. Applying $\xi^r$ to the general
form of the Brianchon--Gram relation~\eqref{eqn:gen_BR}, we obtain
\begin{cor}
   Let $\alpha, \beta : \Cones^d \to \R$ be cone angles and
   $\xi^r = \xi^r(\alpha, \beta)$ for some $0 \leq r \leq d$. Then
   \[
       \widehat{\xi}^r_0(P) - \widehat{\xi}^r_1(P) + \widehat{\xi}^r_2(P) -
       \dots + (-1)^{\dim d} \cdot \widehat{\xi}^r_d(P)
       \ = \ 
        \begin{cases}1 & \text{ if } r = 0\\0& \text{ if } r > 0\end{cases} \, .
   \]
\end{cor}

If $\beta$ is the (unique) cone angle complementary to $\alpha$, we
will simplify the notation and write
$\xi^r(\alpha) \defeq \xi^r(\alpha,\beta)$. When unraveling the
definition of complementary angles, we obtain an equation sometimes
called the Gauss--Bonnet Theorem for polyhedral cones~\cite{AL}.
Let $\chi$ be the Euler characteristic on cones with $\chi(D) = 0$ if $D$ not
a linear subspace and $\chi(D) = (-1)^{\dim D}$ otherwise. 
\begin{lem}\label{lem:xi_zero_one}
    Let $\alpha$ be a cone angle and $C \in \Cones^d$. Then
    \[
        \sum_{r=0}^{d} (-1)^r \xi^r(\alpha)(C) \ = \ \chi(C) \, .
    \]
\end{lem}

\begin{rem}
    The classical spherical intrinsic volumes furthermore sum to
    $1$. This is not necessarily true for generalized spherical
    intrinsic volumes, but one can show using some calculation
    involving Lemma~\ref{lem:AS} that this holds for $\xi(\alpha)$
    when we additionally assume that the cone angle $\alpha$ is
    \emph{even}, that is, that $\alpha(C) = \alpha(-C)$ for all
    polyhedral cones $C \subseteq \R^d$.
\end{rem}

\newcommand{\Gr}{\operatorname{Gr}}%
\newcommand{\gr}{\kappa}%
We will now express the spherical intrinsic volumes in a different basis.
Using integral geometry, we can relate the usual spherical intrinsic volumes
$\nu^r$ to certain integrals over the \Def{Grassmannian} $\Gr^{r, d}$ of
$r$-dimensional linear subspaces in $\R^d$. The \Def{Haar measure} $\mu$
is the unique $O(d)$-invariant measure  on $\Gr^{r, d}$ such that $\mu(\Gr^{r,
d}) = 1$. The \Defn{$r$-th Grassmann-angle} of a cone $C \subset \R^d$ 
\[
   \gr^r(C) \ \defeq \ \mu(\{L \in \Gr^{r, d} : L \cap C = \{0\}\})
\]
was introduced by Grünbaum in~\cite{Grunbaum-grassmann} as a generalization of
interior and exterior angles. Indeed $2 \SCA(C) = 1 - \gr^1(C)$ and $2
\SCA(C^\polar) = \gr^{d-1}(C)$. To generalize the Grassmann angles to arbitrary
valuations we need to shift our point of view. For a fixed pointed cone $C \in
\Cones^d$, define the $\mu$-measurable function $\eps_C : \Gr^{r,d} \to \{0,1\}$
with $\eps_C(L) \defeq 1$ if $C \cap L = \{0\}$ and $0$ otherwise.
The Grassmann-angle can now be expressed as the integral over $\eps_C$
\[
   \gr^r(C) \ = \ \int_{\Gr^{r, d}} \eps_C(L)\,d\mu(L) 
\]
Recall the kinematic formulas for cones.

\begin{thm}[{\cite[Theorem~5.1]{AL}}]
    Let $C \subseteq \R^d$ be a polyhedral cone. Then for $0 \leq r \leq d$ and $1 \leq k \leq d$:
    \begin{align*}
      \int_{\Gr^{r, d}} \nu^k(C \cap L) d\mu(L) \ &= \ \nu^{k + d - r}(C)\,, &
      \int_{\Gr^{r, d}} \nu^0(C \cap L) d\mu(L) \ &= \ \sum_{j=0}^{d-r} \nu^{j}(C)\,.
    \end{align*}
\end{thm}

For a fixed cone, we have almost surely $\eps_C(L) = \chi(C \cap L)$.  From
Lemma~\ref{lem:xi_zero_one} we get 
\[
    \chi(C) = \sum_{i = 0}^d (-1)^i \nu^i(C) \, ,
\]
and we compute
\begin{align*}
  \kappa^r(C)
  \ &= \ \int_{\Gr^{r,d}} \eps_C(L) d\mu(L) \ = \ \int_{\Gr^{r,d}} \chi(C \cap L) d\mu(L) \\
    &= \  \sum_{i = 0}^d (-1)^i \int_{\Gr^{r,d}} \nu^i(C \cap L) d\mu(L) 
    \ = \ \sum_{j = 0}^{d-r}\nu^j(C) + \sum_{i = 1}^d (-1)^i \nu^{i+d-r}(C) \\
    &= \ \sum_{j = 0}^{d-r}\nu^j(C) + \sum_{i = d-r+1}^d (-1)^{i + d - r} \nu^{i}(C)\,.
\end{align*}

This is a slight variation of the usual Crofton-formulas, which better serves
our purposes. We refer to~\cite{AL} for further details. It is not hard to see
that spherical intrinsic volumes and Grassmann-angles encode the same
quantities in a different basis, and conversely we obtain the intrinsic
volumes from the Grassmann-angles as follows:
\[
   \nu^r = \tfrac{1}{2} \big(\gr^{d-r-1} - \gr^{d-r+1}\big)
\]
for $r = 1, \dots, d-1$ as well as
$\nu^0 = \frac{1}{2} (\gr^d + \gr^{d-1})$ and
$\nu^d = \frac{1}{2} (\gr^0 - \gr^1)$.

Using the generalized spherical volumes allows us to give a generalization of
the Grassmann angles, too, by taking the Crofton-formulas as a definition. Thus
we define for any two cone angles $\alpha, \beta : \Cones^d \to \R$ the
\Defn{generalized $r$-th Grassmann-angle} 
\[
    \gr^r(\alpha, \beta) \defeq \sum_{j = 0}^{d - r} \xi^j(\alpha, \beta) +
    \sum_{i = d-r+1}^{d} (-1)^{i-d+r} \xi^i(\alpha, \beta)\,.
\]
we will simplify write $\xi^r(\alpha) = \xi^r(\alpha, \beta)$ and $\gr^r(\alpha)
= \gr^r(\alpha, \beta)$ if $\beta$ is the (unique) complementary angle to
$\alpha$.

As a corollary of
Theorem~\ref{thm:generalized_spherical_intrinsic_volume_is_valuation},
we have:
\begin{cor}
   Every generalized Grassmann-angle is a valuation.
\end{cor}

From this observation we can draw short proofs for most of the results
in Grünbaum's original paper on Grassmann
angles~\cite{Grunbaum-grassmann} where at the same time we replace the
usual Grassmann-angle with our generalized notion $\gr^r = \gr^r(\alpha)$.
\newcommand{\grInt}{\widehat{\gr}}
Let us
 denote by $\grInt_i^r(P)$ the sums of all $r$-th generalized Grassmann
angles of the $i$-faces of a $d$-polytope $P \subseteq \R^d$, that is
\[
    \grInt_i^r(P) \ \defeq \ \sum_{F} \gr^r(T_FP)\,.
\]
We have:


\begin{cor}[Generalization of Grünbaum~{\cite[Theorem~3.3]{Grunbaum-grassmann}}]
   Let $P \subset \R^d$ be a $d$-polytope and
   $\gr^r = \gr^r(\alpha)$ for a cone angle
   $\alpha$ and $0 \leq r \leq d$. Then
   \[
       \sum_{i = 0}^{d - r} (-1)^{i} \cdot \grInt^r_i(P) \ = \ 1\,.
   \]
\end{cor}
\begin{proof}
   Since $\xi^0(\{0\}) = 1$ and
   $\xi^r(\{0\}) = 0$ for all $1 \leq r \leq d$, we have
   $\gr^r(\{0\}) = 1$ for all $0 \leq r \leq d$. Applying $\gr^r$
   to both sides of~\eqref{eqn:gen_BR} yields
   \begin{equation}
       \gr^r(\{0\}) + \sum_{F} (-1)^{\dim F + 1} \gr^r(T_FP) + (-1)^{d+1} \gr^r(\R^d) = 0\,.
       \label{eqn:grassmann_BR}
   \end{equation}
   If $F \subseteq P$ is a face with $\dim F > d-r$, then
   $\xi^j(T_FP) = 0$ for all $j \leq d-r$, as the smallest face of
   $T_FP$ has dimension $\dim F$ and thus the sum in
   \eqref{eqn:gen_sph_int_vol} is empty. Thus, by Lemma~\ref{lem:xi_zero_one}
   \begin{align*}
     \gr^r(T_FP)
     \ = \ \sum_{j = 0}^{d-r} \xi^j(T_FP) + \sum_{i = d-r+1}^d (-1)^{i-d+r} \xi^i(T_FP)
     \ = \ (-1)^{d-r} \sum_{i = 0}^{d} (-1)^{i} \xi^i(T_FP) \ = \ 0\,.
   \end{align*}
   With that, we obtain the claim by rearranging \eqref{eqn:grassmann_BR}.
\end{proof}

In a similar fashion, most of the results in
\cite{Grunbaum-grassmann} can be shown for
generalized Grassmann angles. For example,
\cite[Theorem~3.5]{Grunbaum-grassmann} follows
from an application $\kappa^r$ to Lemma~\ref{lem:AS}.

\bibliographystyle{siam} \bibliography{bibliography.bib}

\begin{thebibliography}{10}

\bibitem{AS15}
{\sc K.~A. Adiprasito and R.~Sanyal}, {\em An {A}lexander-type duality for
  valuations}, Proc. Amer. Math. Soc., 143 (2015), pp.~833--843.

\bibitem{AL}
{\sc D.~Amelunxen and M.~Lotz}, {\em Intrinsic volumes of polyhedral cones: a
  combinatorial perspective}, Discrete Comput. Geom., 58 (2017), pp.~371--409.

\bibitem{Baladze}
{\sc E.~Baladze}, {\em Solution of the {S}z\"okefalvi-{N}agy problem for a
  class of convex polytopes}, Geom. Dedicata, 49 (1994), pp.~25--38.

\bibitem{BGMN05}
{\sc F.~Barthe, O.~Gu\'{e}don, S.~Mendelson, and A.~Naor}, {\em A probabilistic
  approach to the geometry of the {$l^n_p$}-ball}, Ann. Probab., 33 (2005),
  pp.~480--513.

\bibitem{BayerSturmfels}
{\sc M.~Bayer and B.~Sturmfels}, {\em Lawrence polytopes}, Canad. J. Math., 42
  (1990), pp.~62--79.

\bibitem{BB}
{\sc M.~M. Bayer and L.~J. Billera}, {\em Generalized {D}ehn-{S}ommerville
  relations for polytopes, spheres and {E}ulerian partially ordered sets},
  Invent. Math., 79 (1985), pp.~143--157.

\bibitem{crt}
{\sc M.~Beck and R.~Sanyal}, {\em Combinatorial reciprocity theorems}, vol.~195
  of Graduate Studies in Mathematics, American Mathematical Society,
  Providence, RI, 2018.

\bibitem{BER}
{\sc L.~J. Billera, R.~Ehrenborg, and M.~Readdy}, {\em The {$cd$}-index of
  zonotopes and arrangements}, in Mathematical essays in honor of
  {G}ian-{C}arlo {R}ota ({C}ambridge, {MA}, 1996), vol.~161 of Progr. Math.,
  Birkh\"auser Boston, Boston, MA, 1998, pp.~23--40.

\bibitem{BilleraHetyei}
{\sc L.~J. Billera and G.~Hetyei}, {\em Linear inequalities for flags in graded
  partially ordered sets}, J. Combin. Theory Ser. A, 89 (2000), pp.~77--104.

\bibitem{Bloch}
{\sc E.~D. Bloch}, {\em Critical points and the angle defect}, Geom. Dedicata,
  109 (2004), pp.~121--137.

\bibitem{Bolker}
{\sc E.~D. Bolker}, {\em A class of convex bodies}, Trans. Amer. Math. Soc.,
  145 (1969), pp.~323--345.

\bibitem{Ehrenborg-hopf}
{\sc R.~Ehrenborg}, {\em On posets and {H}opf algebras}, Adv. Math., 119
  (1996), pp.~1--25.

\bibitem{GoodeyWeil}
{\sc P.~Goodey and W.~Weil}, {\em Zonoids and generalisations}, in Handbook of
  convex geometry, {V}ol.\ {A}, {B}, North-Holland, Amsterdam, 1993,
  pp.~1297--1326.

\bibitem{GZ}
{\sc C.~Greene and T.~Zaslavsky}, {\em On the interpretation of {W}hitney
  numbers through arrangements of hyperplanes, zonotopes, non-{R}adon
  partitions, and orientations of graphs}, Trans. Amer. Math. Soc., 280 (1983),
  pp.~97--126.

\bibitem{gromer}
{\sc H.~Groemer}, {\em On the extension of additive functionals on classes of
  convex sets}, Pacific J. Math., 75 (1978), pp.~397--410.

\bibitem{GM87}
{\sc M.~Gromov and V.~D. Milman}, {\em Generalization of the spherical
  isoperimetric inequality to uniformly convex {B}anach spaces}, Compositio
  Math., 62 (1987), pp.~263--282.

\bibitem{Grunbaum-grassmann}
{\sc B.~Gr\"{u}nbaum}, {\em Grassmann angles of convex polytopes}, Acta Math.,
  121 (1968), pp.~293--302.

\bibitem{gruenbaum}
{\sc B.~Gr{\"u}nbaum}, {\em {Convex {P}olytopes}}, vol.~221 of {Graduate Texts
  in Mathematics}, Springer-Verlag, New York, second~ed., 2003.
\newblock Prepared and with a preface by Volker Kaibel, Victor Klee, and
  G{\"u}nter M.\ Ziegler.

\bibitem{GrunbaumShephard}
{\sc B.~Gr\"{u}nbaum and G.~C. Shephard}, {\em Descartes' theorem in {$n$}
  dimensions}, Enseign. Math. (2), 37 (1991), pp.~11--15.

\bibitem{hadwiger}
{\sc H.~Hadwiger}, {\em Vorlesungen \"{u}ber {I}nhalt, {O}berfl\"{a}che und
  {I}soperimetrie}, Springer-Verlag, Berlin-G\"{o}ttingen-Heidelberg, 1957.

\bibitem{hohn}
{\sc W.~H\"{o}hn}, {\em Winkel und Winkelsumme im n-dimensionalen euklidischen
  Simplex}, PhD thesis, ETH Zurich, 1953.

\bibitem{RotaKlain}
{\sc D.~A. Klain and G.-C. Rota}, {\em Introduction to geometric probability},
  Lezioni Lincee. [Lincei Lectures], Cambridge University Press, Cambridge,
  1997.

\bibitem{KlivansSwartz}
{\sc C.~J. Klivans and E.~Swartz}, {\em Projection volumes of hyperplane
  arrangements}, Discrete Comput. Geom., 46 (2011), pp.~417--426.

\bibitem{lawrence}
{\sc J.~Lawrence}, {\em Polytope volume computation}, Math. Comp., 57 (1991),
  pp.~259--271.

\bibitem{McM-angle}
{\sc P.~McMullen}, {\em Non-linear angle-sum relations for polyhedral cones and
  polytopes}, Math. Proc. Cambridge Philos. Soc., 78 (1975), pp.~247--261.

\bibitem{McM-polytopealgebra}
{\sc P.~McMullen}, {\em The polytope algebra}, Adv. Math., 78 (1989),
  pp.~76--130.

\bibitem{Naor07}
{\sc A.~Naor}, {\em The surface measure and cone measure on the sphere of
  {$l_p^n$}}, Trans. Amer. Math. Soc., 359 (2007), pp.~1045--1079.

\bibitem{NR03}
{\sc A.~Naor and D.~Romik}, {\em Projecting the surface measure of the sphere
  of {$\ell_p^n$}}, Ann. Inst. H. Poincar\'{e} Probab. Statist., 39 (2003),
  pp.~241--261.

\bibitem{NPS}
{\sc I.~Novik, A.~Postnikov, and B.~Sturmfels}, {\em Syzygies of oriented
  matroids}, Duke Math. J., 111 (2002), pp.~287--317.

\bibitem{PS67}
{\sc M.~A. Perles and G.~C. Shephard}, {\em Angle sums of convex polytopes},
  Math. Scand., 21 (1967), pp.~199--218 (1969).

\bibitem{Sallee}
{\sc G.~T. Sallee}, {\em Polytopes, valuations, and the {E}uler relation},
  Canadian Journal of Mathematics, 20 (1968), pp.~1412--1424.

\bibitem{Schneider}
{\sc R.~Schneider}, {\em Convex bodies: the {B}runn-{M}inkowski theory},
  vol.~151 of Encyclopedia of Mathematics and its Applications, Cambridge
  University Press, Cambridge, expanded~ed., 2014.

\bibitem{Schneider17}
{\sc R.~Schneider}, {\em Combinatorial identities for polyhedral cones},
  Algebra i Analiz, 29 (2017), pp.~279--295.

\bibitem{Schneider-GaussBonnet}
{\sc R.~Schneider}, {\em Polyhedral {G}auss-{B}onnet theorems and valuations},
  Beitr. Algebra Geom., 59 (2018), pp.~199--210.

\bibitem{Schneider_ConvexCones}
\leavevmode\vrule height 2pt depth -1.6pt width 23pt, {\em Convex
  cones---geometry and probability}, vol.~2319 of Lecture Notes in Mathematics,
  Springer, Cham, [2022] \copyright 2022.

\bibitem{SchneiderWeil}
{\sc R.~Schneider and W.~Weil}, {\em Stochastic and integral geometry},
  Probability and its Applications (New York), Springer-Verlag, Berlin, 2008.

\bibitem{shephard-elem}
{\sc G.~C. Shephard}, {\em An elementary proof of {G}ram's theorem for convex
  polytopes}, Canadian J. Math., 19 (1967), pp.~1214--1217.

\bibitem{Shephard-angle}
\leavevmode\vrule height 2pt depth -1.6pt width 23pt, {\em Angle deficiencies
  of convex polytopes}, J. London Math. Soc., 43 (1968), pp.~325--336.

\bibitem{Stanley-eulerian}
{\sc R.~P. Stanley}, {\em A survey of {E}ulerian posets}, in Polytopes:
  abstract, convex and computational ({S}carborough, {ON}, 1993), vol.~440 of
  NATO Adv. Sci. Inst. Ser. C Math. Phys. Sci., Kluwer Acad. Publ., Dordrecht,
  1994, pp.~301--333.

\bibitem{EC1}
\leavevmode\vrule height 2pt depth -1.6pt width 23pt, {\em Enumerative
  combinatorics. {V}olume 1}, vol.~49 of Cambridge Studies in Advanced
  Mathematics, Cambridge University Press, Cambridge, second~ed., 2012.

\bibitem{volland}
{\sc W.~Volland}, {\em Ein {F}ortsetzungssatz f\"ur additive
  {E}ipolyederfunktionale im euklidischen {R}aum}, Arch. Math. (Basel), 8
  (1957), pp.~144--149.

\bibitem{Welzl}
{\sc E.~Welzl}, {\em Gram's equation---a probabilistic proof}, in Results and
  trends in theoretical computer science ({G}raz, 1994), vol.~812 of Lecture
  Notes in Comput. Sci., Springer, Berlin, 1994, pp.~422--424.

\bibitem{White-matroids}
{\sc N.~White}, ed., {\em Theory of matroids}, vol.~26 of Encyclopedia of
  Mathematics and its Applications, Cambridge University Press, Cambridge,
  1986.

\bibitem{Zaslavsky}
{\sc T.~Zaslavsky}, {\em Facing up to arrangements: face-count formulas for
  partitions of space by hyperplanes}, Mem. Amer. Math. Soc., 1 (1975),
  pp.~vii+102.

\bibitem{ziegler}
{\sc G.~M. Ziegler}, {\em Lectures on polytopes}, vol.~152 of Graduate Texts in
  Mathematics, Springer-Verlag, New York, 1995.

\end{thebibliography}

\end{document}